\documentclass[11pt]{article}

\usepackage{amsmath,amssymb,amsthm}
\usepackage{graphics,pstricks,pst-node,xy}
\usepackage{stmaryrd}
\xyoption{all}

\numberwithin{equation}{section}
\newtheorem{thm}{Theorem}[section] 
\newtheorem{prp}[thm]{Proposition}
\newtheorem{lmm}[thm]{Lemma}

\newtheorem{mythm}{Theorem}
\newtheorem{mycrl}[mythm]{Corollary}

\def\BE#1{\begin{equation}\label{#1}}  
\def\EE{\end{equation}}
\def\iv{^{-1}}
 \def\={\;=\;}  \def\+{\,+\,} 
\def\eset{\emptyset} 

\def\lan{\langle}
\def\ran{\rangle}
\def\lr#1{\lan#1\ran}
\def\blr#1{\big\lan#1\big\ran}
\def\lrbr#1{\llbracket{#1}\rrbracket}
\def\ov#1{\overline{#1}}
\def\ti#1{\tilde{#1}}
\def\wt#1{\widetilde{#1}}
\def\e_ref#1{(\ref{#1})}
\def\smsize#1{\begin{small}#1\end{small}}

\def\sf#1{\textsf{#1}}

\def\lra{\longrightarrow}
\def\Lra{\Longrightarrow}
\def\Llra{\Longleftrightarrow}

\def\al{\alpha}

\def\eps{\varepsilon}
\def\ga{\gamma}
\def\de{\delta}

\def\la{\lambda}
\def\si{\sigma}
\def\th{\theta}

\def\a{\mathbf{a}}
\def\nd{\textnormal{d}}

\def\Ga{\Gamma}

\def\ne{\textnormal{e}}

\def\eset{\emptyset}
\def\i{\infty}
\def\hb{\hbar}

\def\cA{\mathcal A}
\def\bA{\mathbb A}
\def\bB{\mathbb B}
\def\cB{\mathcal B}
\def\C{\mathbb C}
\def\cC{\mathcal C}

\def\hD{\widetilde D_w}

\def\E{\mathbf e}
\def\bF{\mathbb F}
\def\F{\mathcal F}
\def\hF{\widetilde\F}\def\oF{\overline\F}

\def\H{\mathcal H}
\def\nH{\textnormal{H}}
\def\I{\mathfrak i}
\def\cI{\mathcal I}

\def\K{\mathcal K}

\def\cL{\mathcal L}
\def\fL{\mathfrak L}

\def\M{\mathfrak M}
\def\bM{\mathbf M}

\def\P{\mathbb P}

\def\cP{\mathcal P}

\def\fR{\mathfrak R}
\def\cS{\mathcal S}
\def\O{\mathcal O}
\def\Q{\mathbb Q}
\def\iQ{\ti\Q_i[\al]^{S_{n-1}}}
\def\hQ{\iQ_{\hb}}
\def\T{\mathbb T}
\def\U{\mathfrak U}
\def\V{\mathcal V}
\def\cY{\mathcal Y}
\def\Y{\mathbb Y}
\def\Z{\mathbb Z}
\def\cZ{\mathcal Z}
\def\X{\mathfrak X}

\def\Edg{\textnormal{Edg}}

\def\ev{\textnormal{ev}}

\def\Span{\textnormal{Span}}

\begin{document}

\thispagestyle{empty}

\title{The Genus One Gromov-Witten Invariants\\
of Calabi-Yau Complete Intersections}
\author{Alexandra Popa\thanks{Partially supported by DMS grant 0846978}}
\date{\today}
\maketitle

\begin{abstract}
\noindent
We obtain mirror formulas for the genus 1  Gromov-Witten invariants of projective Calabi-Yau
complete intersections.
We follow the approach previously used for projective hypersurfaces by extending
the scope of its algebraic results; there is little change in the geometric aspects.
As an application, we check the genus 1 BPS integrality predictions in low degrees
for all projective complete intersections of dimensions 3, 4, and~5.
\end{abstract}

\tableofcontents

\section{Mirror Symmetry Formulas}
\label{mirsym_sec}

Gromov-Witten invariants of projective varieties are counts of curves that are conjectured
(and known in some cases) to possess a rich structure.
The original mirror prediction of \cite{CaDGP} for the genus 0 GW-invariants of a quintic
threefold has since been verified and shown to be a special case of mirror formulas satisfied
by GW-invariants of complete intersections; see \cite{Gi} and \cite{LLY}.
Mirror formulas for the genus 1 GW-invariants of projective Calabi-Yau 
hypersurfaces are obtained
in \cite{g1diff} and  \cite{bcov1}, in particular confirming the prediction of \cite{BCOV}
for a quintic threefold.
In this paper, we obtain mirror formulas for the genus 1 GW-invariants of all projective 
Calabi-Yau complete intersections following the approach in \cite{bcov1},
extending \cite{ZaZ}, and using \cite{bcov0_ci} in place of~\cite{bcov0}.

Throughout this paper, $n,a_1,a_2,\ldots,a_l\ge2$ will be fixed 
integers.\footnote{The assumption that $a_k\!\neq\!1$ is used only to 
streamline the computations in Section~\ref{alg_sec}; 
Theorem~\ref{main_thm0} below is valid as long as $a_k\!\ge\!1$.} 
Let 
\begin{gather*}
\a\equiv(a_1,a_2,\ldots,a_l),\qquad \lr{\a}\equiv\prod_{k=1}^la_k,
\qquad\textnormal{and}\qquad
\a^{\a}\equiv\prod_{k=1}^la_k^{a_k}\,.
\end{gather*}
Let $\eps_0(\a)$ and $\eps_1(\a)$ be the coefficients of 
$w^{n-1-l}$ and $w^{n-2-l}$, respectively, in the power series expansion
of $\frac{(1+w)^n}{\prod\limits_{r=1}^l(1+a_rw)}$ around $w\!=\!0$.
We denote by $X_{\a}$ a smooth complete intersection in $\P^{n-1}$
of multi-degree~$\a$.
This complete intersection is Calabi-Yau  if and only if $\sum\limits_{r=1}^l a_r=n$;
from now on it will be assumed that this condition holds.
Let $N_1^d(X_{\a})$ denote the degree~$d$ genus~1 GW-invariant of~$X_{\a}$.
Note that $\eps_0(\a)$ and $\eps_1(\a)$ describe the top two Chern classes of~$X_{\a}$:
\begin{equation}\label{eps01_e}
c_{n-1-l}(X_{\a})=\eps_0(\a)\nH^{n-1-l}|_{X_{\a}}\,,\qquad
  c_{n-2-l}(X_{\a})=\eps_1(\a)\nH^{n-2-l}|_{X_{\a}}\,,
\end{equation}
where $\nH\in H^2(\P^{n-1})$ is the hyperplane class.

As in \cite{ZaZ}, we denote by
$$\cP\subset 1+q\Q(w)\big[\big[q\big]\big]$$
the subgroup of power series in $q$ with constant term~1 whose coefficients 
are rational functions in $w$ which are holomorphic at $w=0$. Thus, the evaluation map
$$\cP\to 1+q\Q\big[\big[q\big]\big], \qquad F(w,q)\mapsto F(0,q)\,,$$
is well-defined.
We define a map $\bM:\cP\to\cP$ by
\begin{equation}\label{Mdfn}
\bM F(w,q)\equiv\left\{1+\frac{q}{w}\frac{\nd}{\nd q}\right\}\frac{F(w,q)}{F(0,q)}.
\end{equation}
Let $\ti\F\in\cP$ be the hypergeometric series
\begin{equation}\label{F0def} 
\ti\F(w,q)\equiv \sum_{d=0}^{\i}q^d
\frac{\prod\limits_{k=1}^l\prod\limits_{r=1}^{a_kd}(a_kw\!+\!r)}
{\prod\limits_{r=1}^d(w\!+\!r)^n}\,.
\end{equation}
For $p=0,1,\ldots,n-1$, set
\begin{equation}\label{Idfn}
I_p(q)\equiv \bM^p\ti\F(0,q).
\end{equation}
For example,
\begin{equation}\label{I0dfn_e}
I_0(q)=\sum\limits_{d=0}^{\i}q^d\frac{(a_1d)!(a_2d)!\ldots(a_ld)!}{(d!)^n}.
\end{equation}
Let 
\begin{equation}\label{mirmap_e}
J(q)\equiv\frac{1}{I_0(q)}\left\{
\sum_{d=1}^{\i}q^d
\frac{\prod\limits_{k=1}^l(a_kd)!}{(d!)^n}
\left(\sum_{k=1}^{l}\sum_{r=d+1}^{a_kd}\frac{a_k}{r}\right)\right\}\quad
\hbox{and}\quad
Q\equiv q\,\ne^{J(q)}.
\end{equation}
The map $q\!\lra\!Q$ is a change of variables;
it will be called the \sf{mirror map}.

\begin{mythm}\label{main_thm0}
The genus~1 GW-invariants of a multi-degree~$\a$ 
CY CI~$X_{\a}$ in $\P^{n-1}$ are given~by:
\begin{equation*}\begin{split}
\sum_{d=1}^{\i}N^d_1(X_{\a})\,Q^d=&\frac{\lr{\a}}{24}\,\eps_0(\a)\,\left[\log I_0(q)\right]
+\frac{\lr{\a}}{24}\,\eps_1(\a)\,J(q)\\
&\quad-\begin{cases}
\frac{n-l}{48}\log\left(1\!-\!\a^{\a}q\right)+
\sum\limits_{p=0}^{\frac{n-2-l}{2}}\frac{(n-l-2p)^2}{8}\log I_p(q),
&\hbox{if}~2|(n\!-\!l),\\
\frac{n-3-l}{48}\log\left(1\!-\!\a^{\a}q\right)+
\sum\limits_{p=0}^{\frac{n-3-l}{2}}\frac{(n-l-2p)^2-1}{8}\log I_p(q),
&\hbox{if}~2\!\not|(n\!-\!l),
\end{cases}\end{split}\end{equation*}
where $Q\equiv q\,\ne^{J(q)}$. 
\end{mythm}

Since dropping a component of $\a$ equal to 1 has no effect on 
the power series $\ti\F$ in \e_ref{F0def},
this also has no effect on the right-hand side of the formula in Theorem~\ref{main_thm0}
as expected from the relation
$$N_1^d\big(X_{1,a_1,a_2,\ldots,a_l}\big)=N_1^d\big(X_{a_1,a_2,\ldots,a_l}\big).$$
If $l\!=\!1$ and thus $\a\!=\!(n)$, by the Residue Theorem on $S^2$
\begin{equation*}\begin{split}
\eps_0(\a)&=\fR_{w=0}\left\{
 \frac{(1\!+\!w)^n}{w^{n-1}(1\!+\!nw)}\right\}
=-\fR_{w=-1/n}\left\{\frac{(1\!+\!w)^n}{w^{n-1}(1\!+\!nw)}\right\}
 +\fR_{z=0}\left\{\frac{(z\!+\!1)^n}{z^2(z\!+\!n)}\right\}\\
&=\frac{(n\!-\!1)^n}{(-1)^nn^2}+1-\frac{1}{n^2}
=\frac{n^2\!-\!1+(1\!-\!n)^n}{n^2} \,,\\
\eps_1(\a)&=\fR_{w=0}\left\{
 \frac{(1\!+\!w)^n}{w^{n-2}(1\!+\!nw)}\right\}
=-\fR_{w=-1/n}\left\{\frac{(1\!+\!w)^n}{w^{n-2}(1\!+\!nw)}\right\}
 +\fR_{z=0}\left\{\frac{(z\!+\!1)^n}{z^3(z\!+\!n)}\right\}\\
&=-\frac{(n\!-\!1)^n}{(-1)^nn^3}+\frac{n\!-\!1}{2}-\frac{1}{n}+\frac{1}{n^3}
=\frac{(n\!-\!2)(n\!+\!1)}{2n}+\frac{1-(1\!-\!n)^n}{n^3}\,,
\end{split}\end{equation*}
where $\fR_{w=w_0}$ denotes the residue at $w\!=\!w_0$.
Thus, the $l\!=\!1$ case of Theorem~\ref{main_thm0} reduces to 
\cite[Theorem~2]{g1diff}.
The cases $l\!=\!n$ ($X_{\a}\!=\!\eset$) and $l\!=\!n\!-\!1$
($X_{\a}$ is $\lr{\a}\!=\!2$ points) reduce to the case $l\!=\!1$;
as explained in \cite[Section~0.3]{bcov1}
the right-hand side of the formula in Theorem~\ref{main_thm0}
vanishes as expected.

If $l\!=\!n\!-\!2$, $X_{\a}$ is a torus, either $X_3\!\subset\!\P^2$ or 
$X_{2,2}\!\subset\!\P^3$.
In this case, $N_1^d(X_{\a})$ is the number of degree~$d$ maps from genus~1 
curves to $X_{\a}$ modulo automorphisms of such maps; see \cite[0.2]{KlPa}.
Since any such map is an unramified cover of $X_{\a}$ by a torus, 
it follows that $N_1^d(X_{2,2})$ is~0 unless $d$ is divisible by $4$ 
and $N_1^{4r}(X_{2,2})$ is the number of degree~$r$ covers of $X_{2,2}$
by a torus divided by~$r$. 
Thus, using the formula \cite[(B.12)]{bcov1}
for the number of degree~$r$ unramified covers of a torus,
we obtain:
$$\sum_{d=1}^{\i}N_1^d(X_{2,2})\,Q^d=-\sum_{r=1}^{\i}\log\left(1-Q^{4r}\right).$$
This identity together with Theorem~\ref{main_thm0} implies that
$$\frac{1}{6}J(q)-\frac{1}{24}\log(1-16q)-\frac{1}{2}\log I_0(q)=-\sum_{r=1}^{\i}\log\left(1-Q^{4r}\right).$$
The same argument is applied to $X_3$ in \cite[Section~0.3]{bcov1} to obtain
$$\frac{1}{8}J(q)-\frac{1}{24}\log(1-27q)-\frac{1}{2}\log I_0(q)=-\sum_{r=1}^{\i}\log\left(1-Q^{3r}\right).$$
The latter identity is verified directly in~\cite{Sc};
we expect that similar modular-forms techniques can be used to 
verify the former identity directly as well.

If $l\!=\!n\!-\!3$, $X_{\a}$ is a K3 surface, 
either $X_4\!\subset\!\P^3$, $X_{2,3}\!\subset\!\P^4$, or 
$X_{2,2,2}\!\subset\!\P^5$.
Since 
$$\lr{\a}\eps_0(\a)=\chi(X_{\a})=24 \qquad\hbox{and}\qquad \eps_1(\a)=0,$$
by \e_ref{eps01_e}, 
the right hand-side of the formula in Theorem~\ref{main_thm0} is zero in all 3 cases, 
as expected (all GW-invariants of K3 surfaces vanish).

If $l\!=\!n\!-\!4$, $X_{\a}\!\subset\!\P^{n-1}$ is a CY threefold.
Since CY 3-folds are of a particular interest in GW-theory,
we restate the $l\!=\!n\!-\!4$ case of Theorem~\ref{main_thm0} as a corollary below.
In this case, 
\begin{equation*}\begin{split}
\eps_0(\a)&=\fR_{w=0}\left\{
 \frac{(1\!+\!w)^n}{w^4\prod\limits_{r=1}^l(1\!+\!a_rw)}\right\}\\
 &=\frac{n(n\!-\!1)(n\!-\!2)}{6}-\frac{n(n\!-\!1)}{2}\sum_{r=1}^la_r
 +n\!\!\sum_{r_1=1}^l\sum_{r_2=r_1}^l\!a_{r_1}a_{r_2}
 -\sum_{r_1=1}^l\sum_{r_2=r_1}^l\sum_{r_3=r_2}^l\!a_{r_1}a_{r_2}a_{r_3}\,,\\
\eps_1(\a)&=\fR_{w=0}\left\{
 \frac{(1\!+\!w)^n}{w^3\prod\limits_{r=1}^l(1\!+\!a_rw)}\right\}
=\frac{n(n\!-\!1)}{2}-n\sum_{r=1}^la_r
+\sum_{r_1=1}^l\sum_{r_2=r_1}^l\!a_{r_1}a_{r_2}\,.
\end{split}\end{equation*}

\begin{mycrl}\label{bcov_crl}
The genus~1 GW-invariants of a CY CI threefold $X_{\a}\subset\P^{n-1}$ are given by:
\begin{equation*}\begin{split}
\sum_{d=1}^{\i}N_1^d(X_{\a})\,Q^d&=
\left[-2+\frac{\lr{\a}}{72}\left(n-S_3(\a)\right)\right]\log I_0(q)+\frac{\lr{\a}}{48}\left(S_2(\a)-n\right)J(q)\\
&\quad+ \log\left[I_1(q)^{-\frac{1}{2}}(1-\a^{\a}q)^{-\frac{1}{12}}\right],
\end{split}\end{equation*}
where $S_p(\a)\equiv\sum\limits_{r=1}^la_r^p$ and $Q=q\,\ne^{J(q)}$.
\end{mycrl}

Tables~\ref{BPS3_table}-\ref{BPS5_table} below show low-degree genus~1 
BPS numbers for all CY CI 3, 4 and 5-folds obtained from Theorem~\ref{main_thm0}
using \cite[(34.3)]{MirSym}, \cite[(3)]{KlPa}, and \cite[(0.5)]{PaZ}, 
respectively.\footnote{Genus 1 BPS counts in higher dimensions
are yet to be defined.}
Using computer programs\footnote{based on Aleksey Zinger's programs for hypersurfaces}, 
we verified the predicted integrality of these numbers up to degree 100 
for all CY CI 3, 4, and 5-folds. 
While the degree 1 and 2 genus~1 BPS numbers are~0 as expected,
the degree 3 BPS numbers match the classical Schubert calculus on $G(3,n)$.
It should be possible to obtain the degree~4 numbers using 
the approach of~\cite{ESt}, which provides such numbers for hypersurfaces.

\begin{table}[t]\centering{\renewcommand{\arraystretch}{1.15}
\scalebox{0.85}{
\begin{tabular}{c||c|c|c|c|c}
\hline
d&3&4&5&6&7\\
\hline
$X_5$&609250&3721431625&12129909700200&31147299732677250&71578406022880761750\\
$X_{24}$&2560&17407072&24834612736&23689021707008&19078577926517760\\
$X_{33}$&3402&5520393&4820744484&3163476678678&1798399482469092\\
$X_{223}$&64&265113&198087264&89191834992&32343228035424\\
$X_{2222}$&0&14752&8782848&2672004608&615920502784\\
\hline
\end{tabular}}}
\vspace{1mm}
\caption{Low-degree genus~$1$ BPS numbers for all CY CI~$3$-folds}
\label{BPS3_table}
\vspace{5.5mm}

\centering{\renewcommand{\arraystretch}{1.15}
\scalebox{0.85}{
\begin{tabular}{c||c|c|c|c}
\hline
d&3&4&5&6\\
\hline
$X_6$&2734099200&387176346729900&26873294164654597632&1418722120880095142462400\\
$X_{25}$&9058000&845495712250&20201716419250520&320471504960631822000\\
$X_{34}$&2813440&81906297984&1006848150400512&8707175700941649792\\
$X_{224}$&47104&4277292544&42843921424384&249771462364601344\\
$X_{233}$&53928&1203128235&7776816583356&31624897877254152\\
$X_{2223}$&1024&65526084&338199639552&923753814135936\\
$X_{22222}$&0&3779200&15090827264&27474707200000\\
\hline
\end{tabular}}}
\vspace{1mm}
\caption{Low-degree genus~$1$ BPS numbers for all CY CI~$4$-folds}
\label{BPS4_table}
\vspace{5.5mm}

\centering{\renewcommand{\arraystretch}{1.15}
\scalebox{0.8}{\setlength{\tabcolsep}{1pt}
\begin{tabular}{c||c|c|c|c}
\hline
d&3&4&5&6\\
\hline
$X_7$&26123172457235&81545482364153841075&117498479295762788677099464&126043741686161819224278666855602\\
$X_{26}$&69072837120&101190144588682320&41238110240372421024768&11147640321191212498287799296\\
$X_{35}$&8659735175&4075445624973975&725876976084810684840&88498079911311785027601450\\
$X_{44}$&3950411776&1453445296487936&201129967921550639104&19073323868063994075791360\\
$X_{225}$&254083200&244005174397575&33504170048610349120&2706605385511145151653200\\
$X_{234}$&76664320&22674781508976&1639705524423750144&72802469333317263218688\\
$X_{333}$&39550437&5866761130074&289435387120696044&9086367064035583738332\\
$X_{2224}$&1507328&1349735463168&75612640683245568&2228706944980098304000\\
$X_{2233}$&1532160&357068201643&13410965796358752&278702674357074092928\\
$X_{22223}$&32768&21650838624&622096658307072&8565078595779227136\\
$X_{222222}$&0&1342995456&29080932827136&264415120930570240\\
\hline
\end{tabular}}}
\vspace{1mm}
\caption{Low-degree genus~$1$ BPS numbers for all CY CI~$5$-folds}
\label{BPS5_table}
\end{table}

I would like to express my deep gratitude to Aleksey Zinger 
for explaining \cite{g1diff} and \cite{bcov1} to me,
for proposing the questions answered in this paper, and for his invaluable suggestions.

\section{Outline of the proof}
\label{outline_sec}

We prove Theorem~\ref{main_thm0} following the approach used to prove 
\cite[Theorem~2]{g1diff}.
In particular, we compute the reduced genus~1 GW-invariants 
$N_1^{d;0}(X_{\a})$ of $X_{\a}$ defined in~\cite{g1comp2};
these are related to the standard genus~1 invariants by 
Lemma~\ref{g1diff_lmm} below. 

The genus~1 hyperplane theorem of~\cite{LiZ} and the desingularization
construction of~\cite{VaZ} express the reduced genus~1 GW-invariants
of~$X_{\a}$ in terms of integrals over smooth spaces of maps to~$\P^{n-1}$.
We use this in Section~\ref{local_subs} to package the numbers $N_1^{d;0}(X_{\a})$
into a power series $\X(\al,x,Q)$,
in a formal variable $Q$ and with coefficients in 
the equivariant cohomology of $\P^{n-1}$. 
As $\X(\al,x,Q)$ involves integrals on smooth moduli spaces,
the Atiyah-Bott Localization Theorem \cite{ABo} can be applied as in \cite{bcov1}. 
This leads to Proposition~\ref{equivred_prp} of Section~\ref{equivcomp_subs};
the latter expresses $\X(\al,x,Q)$ in terms of residues of some genus~0 
generating functions.

We extract ``the non-equivariant part" of $\X(\al,x,Q)$ in Section~\ref{alg_sec},
using \cite[Lemma 3.3]{bcov1} and mirror formulas for genus~0 generating functions.
This reduces the problem of computing the numbers~$N_1^{d;0}(X_{\a})$ 
to purely algebraic questions concerning the power series~\e_ref{F0def}.
These are addressed in Section~\ref{hg_sec}, which significantly extends \cite{ZaZ};
this section can be read independently of the rest of the paper.

All cohomology groups in this paper will be with rational coefficients.
We will denote by $[n]$,
whenever $n\!\in\!\Z^{\geq 0}$, the set of positive integers not exceeding~$n$:
$$[n]\equiv\big\{1,2,\ldots,n\big\}.$$
Whenever $g$, $d$, $k$ and $n$ are nonnegative integers and $X$ is a smooth subvariety of~$\P^{n-1}$, 
$\ov\M_{g,k}(X,d)$ will denote the moduli space of stable degree~$d$ maps into $X$
from genus~$g$ curves with $k$ marked points
and 
$$\ev_i\!:\ov\M_{g,k}(\P^{n-1},d)\quad [\cC,y_1,\ldots,y_k,f]\lra f(y_i), \qquad i=1,2,\ldots,k,$$
for the evaluation map at the $i$-th marked point; see \cite[Chapter~24]{MirSym}.
For each $m\!\in\!\Z^{>0}$, define
\begin{gather*}
\ov\M_{(m)}(X,d)\equiv
\left\{(b_i)_{i\in[m]}\in\prod_{i=1}^m\ov\M_{0,1}(X,d_i)\!:
d_i\!\in\!\Z^{>0},
\sum_{i=1}^m d_i\!=\!d,~
\ev_1(b_i)\!=\!\ev_1(b_{i'})~\forall\, i,i'\!\in\![m]\right\},\\
\ev_1\!: \ov\M_{(m)}(X,d)\lra X, \qquad (b_i)_{i\in[m]}\lra \ev_1(b_i),
\end{gather*}
where $i$ is any element of $[m]$.
For each $i\!\in\![m]$, let 
$$\pi_i\!: \ov\M_{(m)}(X,d)\lra 
\bigsqcup_{d_i\in \Z^{>0},\,\,d_i\leq d}\ov\M_{0,1}(X,d_i)$$
be the projection onto the $i$-th component.
If $p\!\in\!\Z^{\ge0}$, we define $\eta_p\in H^{2p}(\ov\M_{(m)}(X,d))$
to be the degree~$2p$ term of 
$$\prod_{i=1}^m\pi_i^*\frac{1}{1-\psi_1}\in 
H^*\big(\ov\M_{(m)}(X,d)\big).$$
Thus, $\eta_p$ is the sum of all degree~$p$ monomials in 
$\big\{\pi_i^*\psi_1\!:~i\!\in\![m]\big\}.$

The symmetric group on $m$ elements, $S_m$, acts on $\ov\M_{(m)}(X,d)$
by permuting the elements of each $m$-tuple of stable maps.
Let 
$$\cZ_{(m)}(X,d)\equiv\ov\M_{(m)}(X,d)\big/S_m.$$
Since the map $\ev_1$ and the cohomology class $\eta_p$
on $\ov\M_{(m)}(X,d)$ are $S_m$-invariant, they descend to the quotient:
$$\ev_1\!:  \cZ_{(m)}(X,d)\lra X \qquad\hbox{and}\qquad
\eta_p\in H^{2p}\big(\cZ_{(m)}(X,d)\big).$$

Let $\U$ be the universal curve over $\ov\M_{(m)}(\P^{n-1},d)$,
with structure map~$\pi$ and evaluation map~$\ev$:
$$\xymatrix{\U \ar[d]^{\pi} \ar[r]^{\ev} & \P^{n-1} \\
\ov\M_{(m)}(\P^{n-1},d).}$$
The orbi-sheaf
$$\pi_*\ev^*\bigoplus\limits_{r=1}^l\O_{\P^{n-1}}(a_r)\lra\ov\M_{(m)}(\P^{n-1},d)$$
is locally free; it~is the sheaf of (holomorphic) sections of 
the vector orbi-bundle
$$\V_{(m)}\equiv \ov\M_{(m)}(\cL,d)
\lra\ov\M_{(m)}(\P^{n-1},d),$$
where $\cL\!\lra\!\P^{n-1}$ is the total space 
of the vector bundle corresponding to the sheaf~$\bigoplus\limits_{r=1}^l\O_{\P^{n-1}}(a_r)$.
By the (genus-zero) hyperplane-section relation,
\begin{equation}\label{HPrel_e}
\blr{\eta_{p-2m}\ev_1^*\nH^{n-1-l-p}\!,\!\left[\cZ_{(m)}\left(X_{\a},d\right)\right]^{vir}}
\!=\!\frac{1}{m!}\blr{\eta_{p-2m}\ev_1^*\nH^{n-1-l-p}\!e(\V_{(m)}),
\left[\ov\M_{(m)}(\P^{n-1},d)\right]},
\end{equation}
where $\nH\in H^2(\P^{n-1})$ is the hyperplane class.

There is a natural surjective bundle homomorphism
$$\wt\ev_1\!:\V_{(1)}\lra\ev_1^*\bigoplus\limits_{r=1}^l\O_{\P^{n-1}}(a_r),\qquad
\left([\cC,u,\xi]\right)\lra \xi\left(x_1(\cC)\right),$$
$\ov\M_{(1)}(\P^{n-1},d)\!\equiv\!\ov\M_{0,1}(\P^{n-1},d)$,
where $x_1(\cC)$ is the marked point.
Thus,
$$\V_{(1)}'\equiv\ker\wt\ev_1\lra\ov\M_{(1)}(\P^{n-1},d)$$
is a vector 
orbi-bundle.\footnote{In the notation of Section~\ref{local_subs}, 
$\V_{(1)}=\V_0$ and $\V'_{(1)}=\V_0'$.}
It is straightforward to see~that
\begin{equation}\label{vbsplit_e}
e(\V_{(m)})=\lr{\a}\,\ev_1^*\nH^l\prod_{i=1}^m\pi_i^*e(\V_{(1)}').
\end{equation}

If $f\!=\!f(w)$ admits a Laurent series expansion around $w\!=\!0$, 
for any $p\!\in\!\Z$ we denote by $\lrbr{f(w)}_{w;p}$ the coefficient of~$w^p$. 
Let
\begin{equation}\label{coeff_e}
\left\llbracket\sum_{d=0}^{\i} f_d(w)\, Q^d\right\rrbracket_{w;p}
\equiv\sum_{d=0}^{\i}\left\llbracket f_d(w)\right\rrbracket_{w;p}\,Q^d
\qquad\text{if}\quad f_d\in\Q(w)~\forall\,d\ge0.
\end{equation}

Theorem~\ref{main_thm0} follows immediately from from
\e_ref{derivgen_e2},  Theorem~\ref{main_thm} stated at the beginning
of Section~\ref{alg_sec}, and Lemma~\ref{g1diff_lmm} below,
which extends \cite[Lemma~2.2]{g1diff} to complete intersections.

\begin{lmm}\label{g1diff_lmm}
If $X_{\a}\!\subset\!\P^{n-1}$ is a complete intersection of multi-degree $\a$, 
\begin{equation*}\begin{split}
&N_1^d(X_{\a})=N_1^{d;0}(X_{\a})\\
&\hspace{.75in}+
\frac{1}{24}\sum_{p=2}^{n-1-l}\sum_{m=1}^{2m\le p}
(-1)^m(m\!-\!1)!\,
\blr{\eta_{p-2m}\ev_1^*\left(c_{n-1-l-p}(X_{\a})\right),\left[\cZ_{(m)}\left(X_{\a},d\right)\right]^{vir}}.
\end{split}\end{equation*}
Furthermore, for all $p\!\in\!\Z^{\ge0}$
\BE{chclass_e}
c_p(X_{\a})=\left\llbracket\frac{(1\!+\!w)^n}{\prod\limits_{r=1}^l(1\!+\!a_rw)}\right\rrbracket_{w;p}\nH^p\Big|_{X_{\a}}\EE
and for all $p\!\le\!n\!-\!1\!-\!l$
\begin{equation*}\begin{split}
&\sum_{d=1}^{\i}Q^d\left(
\sum_{m=1}^{2m\le p}(-1)^m(m\!-\!1)!
\blr{\eta_{p-2m}\ev_1^*\nH^{n-1-l-p}\Big|_{X_{\a}},\left[\cZ_{(m)}\left(X_{\a},d\right)\right]^{vir}}\right)=
-\lr{\a}\left\llbracket\log\frac{\ti\F(w,q)}{I_0(q)}\right\rrbracket_{w;p},\qquad\qquad\,\,
\end{split}\end{equation*}
where $\nH\in H^2(\P^{n-1})$ is the hyperplane class,
$Q$ and $q$ are related by the mirror map~\e_ref{mirmap_e},
and $\ti\F(w,q)$ and $I_0(q)$ are given  by
by~\e_ref{F0def} and \e_ref{I0dfn_e}, respectively.
\end{lmm}

\begin{proof}
The first identity above is a special case of \cite[(2.15)]{g1diff}.
The second identity is immediate from
$$c(T\P^{n-1})=(1+\nH)^n \qquad\hbox{and}\qquad
c\big(N_{X_{\a}/\P^{n-1}}\big)=\prod_{r=1}^l\left(1+a_r\nH\Big|_{X_{\a}}\right).$$
It remains to verify the third identity.
For each $r\!\in\!\Z^{\ge0}$, let
$$Z_r(Q)\equiv\sum_{d=1}^{\i}Q^d\blr{\psi_1^r\ev_1^*\nH^{n-3-r}\,e(\V_{(1)}'),
[\ov\M_{(1)}(\P^{n-1},d)]}.$$
By~\e_ref{HPrel_e},~\e_ref{vbsplit_e}, 
and the decomposition along the small diagonal 
in $(\P^{n-1})^m$, the left-hand side of the third identity 
in Lemma~\ref{g1diff_lmm} above equals
\begin{equation*}\begin{split}
&\lr{\a}\sum_{d=1}^{\i}Q^d\left(\sum_{m=1}^{2m\le p}\frac{(-1)^m}{m}
\blr{\ev_1^*\nH^{n-1-p}
\prod_{i=1}^m\pi_i^*\frac{e(\V_{(1)}')}{1\!-\!\psi_1},
\left[\ov\M_{(m)}(\P^{n-1},d)\right]}\right)
\end{split}\end{equation*}

\begin{equation*}\begin{split}
&\qquad
=\lr{\a}\sum_{m=1}^{2m\le p}\frac{(-1)^m}{m}
\sum_{d=1}^{\i}Q^d\sum_{\underset{d_i>0}{\sum\limits_{i=1}^m\!d_i=d}}
\sum_{\underset{p_i\ge0}{\sum\limits_{i=1}^m\!p_i=p}}\prod_{i=1}^m
\blr{\ev_1^*\nH^{n-1-p_i}\frac{e(\V_{(1)}')}{1\!-\!\psi_1},
\left[\ov\M_{(1)}(\P^{n-1},d_i)\right]}\notag\\
&\qquad
=\lr{\a}\sum_{m=1}^{2m\le p}\frac{(-1)^m}{m}
\sum_{\underset{p_i\ge2}{\sum\limits_{i=1}^m\!p_i=p}}
\prod_{i=1}^m\!Z_{p_i-2}(Q)
=-\lr{\a}\left\llbracket\log\left(1+\sum_{r=0}^{n-3}\!Z_r(Q)w^{r+2}\right)
\right\rrbracket_{w;p}.
\end{split}\end{equation*}
The third statement of Lemma~\ref{g1diff_lmm} now follows from
\begin{equation*}
1+\sum_{r=0}^{n-3}Z_r(Q)w^{r+2}
=\ne^{-J(q)w}\frac{\ti\F(w,q)}{I_0(q)}\in\Q[w]\big[\big[q\big]\big]/w^n\,;
\end{equation*}
the last identity is obtained from \cite[Theorem~11.8]{Gi}
using  the string relation \cite[Section~26.3]{MirSym}.
\end{proof}

\section{Equivariant Setup} 
\label{introlocal_sec}

\subsection{Equivariant cohomology}
\label{equivcoh_subs}

This section reviews the basics of equivariant cohomology
following \cite[Section 1.1]{bcov1} closely and setting up related notation.

The classifying space for the $n$-torus $\T$ is $B\T\equiv(\P^{\i})^n$.
Thus, the group cohomology of~$\T$ is
$$H_{\T}^*\equiv H^*(B\T)=\Q[\al_1,\ldots,\al_n],$$
where $\al_i\!\equiv\!\pi_i^*c_1(\ga^*)$,
$\ga\!\lra\!\P^{\i}$ is the tautological line bundle,
and $\pi_i\!: (\P^{\i})^n\!\lra\!\P^{\i}$ is
the projection to the $i$-th component.
In the remainder of the paper, 
$$\al=(\al_1,\ldots,\al_n).$$
The field of fractions of $H^*_{\T}$ will be denoted by
$$\Q_{\al}\equiv \Q(\al_1,\ldots,\al_n).$$
We denote the equivariant $\Q$-cohomology of a topological space $M$
with a $\T$-action by $H_{\T}^*(M)$.
If the $\T$-action on $M$ lifts to an action on a complex vector bundle $V\!\lra\!M$,
let $\E(V)\in H_{\T}^*(M)$ denote the \sf{equivariant Euler class of} $V$.
A continuous $\T$-equivariant map $f\!:M\!\lra\!M'$ between two compact oriented 
manifolds induces a pushforward homomorphism
$$f_*\!: H_{\T}^*(M) \lra H_{\T}^*(M'),$$
which is characterized by the property that 
\begin{equation}\label{pushdfn_e}
\int_{M'}(f_*\eta)\,\eta'=\int_M\eta\,(f^*\eta') 
\qquad~\forall~\eta\!\in\! H_{\T}^*(M),\, \eta'\!\in\! H_{\T}^*(M').
\end{equation}
If $M'$ is a point, this is the integration-along-the-fiber homomorphism
$$\int_M\!: H_{\T}^*(M)\lra H_{\T}^*$$
for the fiber bundle $E\T\times_{\T}M\!\lra\!B\T$.

Throughout this paper, $\T$ will act on $\P^{n-1}$ in the standard way:
$$\big(\ne^{\I\th_1},\ldots,\ne^{\I\th_n}\big)\cdot [z_1,\ldots,z_n] 
=\big[\ne^{\I\th_1}z_1,\ldots,\ne^{\I\th_n}z_n\big].$$
This action has $n$~fixed points:
$$P_1=[1,0,\ldots,0], \quad P_2=[0,1,0,\ldots,0], 
\quad\ldots,\quad P_n=[0,\ldots,0,1].$$
For each $i\!=\!1,2,\ldots,n$, let
\begin{equation}\label{phidfn_e}
\phi_i\equiv \prod_{k\neq i}(x\!-\!\al_k) \in H_{\T}^*(\P^{n-1}).
\end{equation}
By the Atiyah-Bott Localization Theorem \cite{ABo}, 
\begin{equation}\label{phiprop_e}
\eta|_{P_i}=\int_{\P^{n-1}}\eta\phi_i \in\Q_{\al}
\quad\forall\,\eta\!\in\!H_{\T}^*(\P^{n-1})\otimes_{H^*_{\T}}\Q_{\al},~i=1,2,\ldots,n;
\end{equation}
thus,  $\phi_i$ is the equivariant Poincar\'{e} dual of $P_i$.

The standard action of $\T$ on $\P^{n-1}$
lifts to an action on the tautological bundle
$$\ga\equiv\O_{\P^{n-1}}(-1)\subset\P^{n-1}\times\C^n$$ 
by restricting the standard diagonal $\T$-action on $\P^{n-1}\times\C^n$.
The \sf{equivariant hyperplane class} is defined to be
$$x\equiv \E(\ga^*)\equiv\E\left(\O_{\P^{n-1}}(1)\right)\in H_{\T}^2(\P^{n-1}).$$
The equivariant cohomology of $\P^{n-1}$ is given by 
\begin{equation}\label{pncoh_e}
H_{\T}^*(\P^{n-1})= \Q[x,\al_1,\ldots,\al_n]\big/(x\!-\!\al_1)\ldots(x\!-\!\al_n).
\end{equation}
The restriction map on the equivariant cohomology induced by the inclusion 
$P_i\!\lra\!\P^{n-1}$ is given~by
\begin{equation}\label{restrmap_e}
H_{\T}^*(\P^{n-1})=\Q[x,\al_1,\ldots,\al_n]\big/\prod\limits_{k=1}^n(x\!-\!\al_k)
\lra H_{\T}^*(P_i)=\Q[\al_1,\ldots,\al_n], \qquad x\lra\al_i,
\end{equation}
and so
\begin{equation}\label{uniquecond_e}
\eta=0 \in H_{\T}^*(\P^{n-1}) \qquad\Llra\qquad
\eta|_{P_i}=0\in H_{\T}^* ~~\forall~i=1,2,\ldots,n.
\end{equation}

\subsection{Generating function for reduced genus 1 GW-invariants}
\label{local_subs}

As in~\cite{bcov1}, the reduced genus~$1$ GW-invariants $N_1^{d;0}(X_{\a})$ 
of $X_{\a}$ are packaged into a generating function~$\X$;
this is a power series in the formal variable $Q$ with coefficients
in the equivariant cohomology of $\P^{n-1}$. 
In this section, we define $\X$ and explain what
its relationship with $N_1^{d;0}(X_{\a})$ is; 
see \e_ref{pushclass_e} and \e_ref{derivgen_e2}. 

Let $\pi\!:\U\!\lra\!\ov\M_{g,k}(\P^{n-1},d)$ be the universal curve 
with  evaluation map~$\ev$ as before and
$$\V_0\lra\ov\M_{0,k}(\P^{n-1},d)$$
the vector bundle corresponding to the locally free sheaf
$$\bigoplus\limits_{r=1}^l
     \pi_*\ev^*\O_{\P^{n-1}}(a_r)\lra\ov\M_{0,k}(\P^{n-1},d).$$ 
The Euler class $e(\V_0)$ relates genus~0 GW-invariants of $X_{\a}\subset\P^{n-1}$
to genus~0 GW-invariants of $\P^{n-1}$;
it also appears in the genus~0 2~point generating functions 
\e_ref{Z1ptdfn_e}-\e_ref{Z2ptdfn_e} 
which are used in the proof in Theorem~\ref{main_thm}.

The genus 1 GW-invariants of $X_{\a}$ are related to the GW-invariants of 
$\P^{n-1}$ in a more complicated way.
This is partly because $\ov\M_{1,k}(\P^{n-1},d)$ is not an orbifold
and
$$\bigoplus\limits_{r=1}^l\pi_*\ev^*\O_{\P^{n-1}}(a_r)
\lra\ov\M_{1,k}(\P^{n-1},d)$$ 
is not locally free.
However, it is shown in~\cite{VaZ} that there exists
a natural desingularization 
$$p\!:\wt\M_{1,k}^0(\P^{n-1},d)\lra\ov\M_{1,k}^0(\P^{n-1},d)$$
of the main component of $\ov\M_{1,k}(\P^{n-1},d)$,
whose generic element is a map from smooth domain.
There is also a vector orbi-bundle $\V_1$ over $\wt\M_{1,k}^0(\P^{n-1},d)$
so that the diagram
\begin{equation*}\label{desing_e}
\xymatrix{\V_1 \ar[d] \ar[r]^>>>>>>>{p_*}&
 \bigoplus\limits_{r=1}^l\pi_*\ev^*\O_{\P^{n-1}}(a_r) \ar[d]\\ 
\wt\M_{1,k}^0(\P^{n-1},d)\ar[r]^p &  \ov\M_{1,k}^0(\P^{n-1},d)}
\end{equation*}
commutes.
By \cite[Theorem~1.1]{LiZ} and \cite[Theorem~1.1]{g1cone}, 
\BE{derivgen_e1} d\,N_1^{d;0}(X_{\a})
=\blr{e(\V_1)\,\ev_1^*{\nH},\big[\wt\M_{1,1}^0(\P^{n-1},d)\big]}.\EE

The standard $\T$-action on $\P^{n-1}$  induces $\T$-actions on the moduli spaces of 
$\ov\M_{g,k}(\P^{n-1},d)$ 
and lifts to an action on $\wt\M_{1,k}^0(\P^{n-1},d)$.
The evaluation maps,
$$\ev_i\!:\ov\M_{g,k}(\P^{n-1},d),\wt\M_{1,k}^0(\P^{n-1},d)\lra\P^{n-1}, 
\quad [\cC,y_1,\ldots,y_k,f]\lra f(y_i), \qquad i\in[k],$$
are $\T$-equivariant.
The natural $\T$-action on $\O_{\P^{n-1}}(-1)\!\lra\!\P^{n-1}$
induces  $\T$-actions on the sheafs $\pi_*\ev^*\O_{\P^{n-1}}(a)$ and the vector bundle
$$\V_1\lra \wt\M_{1,1}^0(\P^{n-1},d).$$
With $\ev_{1,d}$ the evaluation map on $\wt\M_{1,1}^0(\P^{n-1},d)$, let
$$\X(\al,x,Q)\equiv\sum_{d=1}^{\i}Q^d(\ev_{1,d*}\E(\V_1)\big)\in 
\big(H_{\T}^{n-2}(\P^{n-1})\big)\big[\big[Q\big]\big].$$
By~\e_ref{pncoh_e}, 
\begin{equation}\label{pushclass_e}
\X(\al,x,Q)=\X_0(Q)x^{n-2}+\X_1(\al,Q)x^{n-3}+\ldots+\X_{n-2}(\al,Q)x^0,
\end{equation}
for some $\X_0\!\in\!\Q\big[\big[Q\big]\big]$ and power series $\X_p\in\Q[\al_1,\ldots,\al_n]\big[\big[Q\big]\big]$,
whose coefficients are symmetric degree~$p$ homogeneous polynomials 
in $\al_1,\ldots,\al_n$.
By~\e_ref{derivgen_e1} and \e_ref{pushdfn_e}, 
\begin{equation}\label{derivgen_e2}
Q\frac{\nd}{\nd Q}\sum_{d=1}^{\i}N_1^{d;0}(X_{\a})\,Q^d=\X_0(Q).
\end{equation}
By~\e_ref{restrmap_e}, \e_ref{phiprop_e}, and~\e_ref{pushdfn_e}, 
\begin{equation}\label{Frestr_e}\begin{split}
\X(\al,\al_i,Q) &=\X(\al,x,Q)\big|_{P_i}
=\sum_{d=1}^{\i}Q^d\int_{\P^{n-1}}\big(\ev_{1,d*}\E(\V_1)\big)\phi_i\\
&=\sum_{d=1}^{\i}Q^d\int_{\wt\M_{1,1}^0(\P^{n-1},d)}\E(\V_1)\ev_1^*\phi_i
\in\Q_{\al}\big[\big[Q\big]\big]
\end{split}\end{equation}
for each $i\!=\!1,2,\ldots,n$.
Since $\X$ is symmetric in $\al_1,\ldots,\al_n$,
$\X_0\in\Q\big[\big[Q\big]\big]$ is completely  determined by either of the
$n$ power series in~\e_ref{Frestr_e}.
We use this to obtain the explicit formula for $\X_0$ given
in Theorem~\ref{main_thm}.

\subsection{A localization proposition}
\label{equivcomp_subs}

As in \cite{bcov1}, we express $\X(\al,\al_i,Q)$ in terms of residues 
of genus 0 generating functions.
Proposition~\ref{equivred_prp} below is the analogue of 
\cite[Propositions~1.1,~1.2]{bcov1};
its proof is essentially identical to the proof of 
\cite[Propositions~1.1,~1.2]{bcov1} in \cite[Section~2]{bcov1}.
In this section, we set up the notation needed to 
state Proposition~\ref{equivred_prp}, motivate it, and describe 
the few minor changes needed in  \cite[Section~2]{bcov1} for a complete proof
of this proposition.
In the remainder of this paper, we will use Proposition~\ref{equivred_prp}
to obtain an explicit formula for~$\X_0$.

If $f$ is a rational function in $\hb$ and possibly other variables and
$\hb_0\!\in\!S^2$, let $\fR_{\hb=\hb_0}f(\hb)$ denote the residue
of the one-form $f(\hb)d\hb$ at $\hb\!=\!\hb_0$;
thus,
$$\fR_{\hb=\i}f(\hb)\equiv-\fR_{w=0}\big\{w^{-2}f(w^{-1})\big\}.$$
If $f$ involves variables other than $\hb$, $\fR_{\hb=\hb_0}f(\hb)$
is a function of the other variables.
If $f$ is a power series in $Q$ with coefficients that are rational
functions in~$\hb$ and possibly other variables, 
let $\fR_{\hb=\hb_0}f(\hb)$ denote the power series in~$Q$ obtained by 
replacing each of the coefficients by its residue at $\hb\!=\!\hb_0$.
If $\hb_1,\ldots,\hb_k$ is a collection of points in~$S^2$, not necessarily distinct, 
we define
$$\fR_{\hb=\hb_1,\ldots,\hb_k}f(\hb) \equiv
\sum_{z\in\{\hb_1,\ldots,\hb_k\}}\!\!\!\!\!\! \fR_{\hb=z}f(\hb).$$
If $\hb_0\!\in\!\C$ or $\hb_0$ is one of the ``other'' variables in $f$, let
$$\fR_{\hb=-\a\hb_0}f(\hb) \equiv \fR_{\hb=-a_1\hb_0,\ldots,-a_l\hb_0}f(\hb).$$
For instance, if $\a=(2,2,3,3,3,3)$ and $\al_i$ is one of the other variables, then $$\fR_{\hb=-\a\al_i}f(\hb)\equiv\fR_{\hb=-2\al_i}f(\hb)+\fR_{\hb=-3\al_i}f(\hb).$$
 
Since the $\T$-equivariant bundle homomorphism
$$\wt\ev_1\!:\V_0\lra\bigoplus_{r=1}^l\ev_1^*\O_{\P^{n-1}}(a_r),
\qquad [\cC,x_1,\ldots,x_k,f,\xi]\lra\big[\xi(x_1(\cC))\big],$$
is surjective, its kernel
\begin{equation*}\begin{split}
\V_0'\equiv\ker\wt\ev_1\lra\ov\M_{0,k}(\P^{n-1},d)\,,
\end{split}\end{equation*}
is a $\T$-equivariant vector bundle.
Since the $\T$-action on $\ov\M_{g,k}(\P^{n-1},d)$
lifts naturally to the tautological tangent line bundles $L_i$,
there are well-defined equivariant $\psi$-classes
$$\psi_i\equiv c_1(L_i^*)\in H_{\T}^*\big(\ov\M_{g,k}(\P^{n-1},d)\big);$$
see \cite[Section~25.2]{MirSym}.
For all $i,j\!=\!1,2,\ldots,n,$ let
\begin{alignat}{1}
\label{Z1ptdfn_e}
\cZ_i^*(\hb,Q) &\equiv \sum_{d=1}^{\i}Q^d
\int_{\ov\M_{0,2}(\P^{n-1},d)}\frac{\E(\V_0')}{\hb\!-\!\psi_1}\ev_1^*\phi_i;\\
\label{Z1pt2dfn_e}
\cZ_{ij}^*(\hb,Q) &\equiv \hb^{-1}\sum_{d=1}^{\i}Q^d
\int_{\ov\M_{0,2}(\P^{n-1},d)}\frac{\E(\V_0')}{\hb\!-\!\psi_1}
\ev_1^*\phi_i\ev_2^*\phi_j;\\
\label{Z2ptdfn_e}
\wt\cZ_{ij}^*(\hb_1,\hb_2,Q) &\equiv \frac{1}{2\hb_1\hb_2}\sum_{d=1}^{\i}Q^d
\int_{\ov\M_{0,2}(\P^{n-1},d)}
\frac{\E(\V_0')}{(\hb_1\!-\!\psi_1)(\hb_2\!-\!\psi_2)}
\ev_1^*\phi_i\ev_2^*\phi_j.
\end{alignat}
Explicit formulas for these generating functions are given explicitly
in \cite[Theorem~11.8]{Gi} and \cite[Theorem~4]{bcov0_ci}.
These theorems show that in particular 
$$\cZ_i^*,\cZ_{ij}^*\in\Q_{\al}(\hb)\big[\big[Q\big]\big]
\qquad\hbox{and}\qquad \wt\cZ_{ij}^*\in\Q_{\al}(\hb_1,\hb_2)\big[\big[Q\big]\big].$$
Thus, the $\hb$-residues of these power series are well-defined.
Since the $Q$ degree~$0$ term of the power series $\cZ_i^*(\hb,Q)$ is $0$,
the residue
\begin{equation}\label{etadfn_e}
\eta_i(Q)\equiv\fR_{\hb=0}\Big\{\log\big(1\!+\!\cZ_i^*(\hb,Q)\big)\Big\}
\in\Q_{\al}\big[\big[Q\big]\big]
\end{equation}
is well-defined.
Let
\begin{equation}\label{Phi0dfn_e}
\Phi_0(\al_i,Q)\equiv\fR_{\hb=0}\Big\{
\hb^{-1}\ne^{-\eta_i(Q)/\hb}\big(1\!+\!\cZ_i^*(\hb,Q)\big)\Big\}
\in\Q_{\al}\big[\big[Q\big]\big].
\end{equation}
By \cite[Lemma 2.3]{bcov1}, 
the power series $\ne^{-\eta_i(Q)/\hb}(1\!+\!\cZ_i^*(\hb,Q))$ 
is holomorphic at $\hb\!=\!0$;
thus $\Phi_0(\al_i,Q)$ is its value at $\hb\!=\!0$.\footnote{While $l$ is meant to be 1
in \cite[Section~2.2]{bcov1}, the argument goes through for any $\a$
without any change.}
Note that the degree-zero term of $\Phi_0(\al_i,Q)$ is~1.

Proposition~\ref{equivred_prp} below is obtained by applying
the Atiyah-Bott Localization Theorem \cite{ABo} 
to the last expression in~\e_ref{Frestr_e}.
As described in detail in \cite[Sections~1.3,~1.4]{bcov1},
the fixed loci of the $\T$-action on $\wt\M_{1,1}^0(\P^{n-1},d)$ 
are indexed by decorated graphs with one marked point.
The vertices are decorated by elements of~$[n]$, indicating 
the $\T$-fixed point of $\P^{n-1}$ to which the node or component
corresponding to the vertex is mapped~to.
These graphs have either zero loops and one distinguished vertex
(as in Figure~\ref{Bfig}) or one loop (as in Figure~\ref{Afig}), depending 
on whether the stable maps they describe are constant or not 
on the principal component of the domain.\footnote{Figures~\ref{Afig}
and~\ref{Bfig} are Figures~1 and~4 in~\cite{bcov1}; 
they are used by permission to indicate what is involved in the proof 
of (the $l\!=\!1$ case of) Proposition~\ref{equivred_prp} in~\cite{bcov1}.}
The graphs with no loops are called $B$-graphs in \cite{bcov1},
while the graphs with one loop are called $A$-graphs.
In a $B$-graph, the distinguished vertex corresponds to 
the contracted principal component.
As every graph has a marked point, even the $A$-graphs have 
a distinguished vertex: the vertex {\it in} the loop closest
to the vertex to which the marked point is attached.
The distinguished vertices are indicated by thick dots 
in the four graphs in Figures~\ref{Afig} and~\ref{Bfig}.

\begin{figure}
\begin{pspicture}(-.3,-1.1)(10,2)
\psset{unit=.4cm}
\pscircle*(12,0){.3}\rput(12,-.7){\smsize{$3$}}
\psline[linewidth=.04](12,0)(12,3)\rput(11.6,3){\smsize{$\bf 1$}}
\psline[linewidth=.04](12,0)(9,2)\pscircle*(9,2){.2}
\rput(10.6,1.5){\smsize{$2$}}\rput(9.3,2.5){\smsize{$1$}}
\psline[linewidth=.04](12,0)(9,-2)\pscircle*(9,-2){.2}
\rput(10.4,-.5){\smsize{$1$}}\rput(9.3,-2.5){\smsize{$1$}}
\psline[linewidth=.04](5.5,2)(9,2)\pscircle*(5.5,2){.2}
\rput(7.3,2.5){\smsize{$2$}}\rput(5.1,2.5){\smsize{$4$}}
\psline[linewidth=.04](5.5,-2)(9,-2)\pscircle*(5.5,-2){.2}
\rput(7.3,-1.5){\smsize{$3$}}\rput(5.1,-2.5){\smsize{$3$}}
\psline[linewidth=.04](5.5,-2)(5.5,2)\rput(5,0){\smsize{$1$}}
\psline[linewidth=.04](2.5,0)(5.5,-2)\pscircle*(2.5,0){.2}
\rput(3.4,-1.2){\smsize{$2$}}\rput(2,0){\smsize{$1$}}
\psline[linewidth=.04](2.5,0)(2.5,3.5)\pscircle*(2.5,3.5){.2}
\rput(2,1.7){\smsize{$2$}}\rput(2,3.5){\smsize{$2$}}
\psline[linewidth=.04](12,0)(15,2)\pscircle*(15,2){.2}
\rput(13.4,1.5){\smsize{$2$}}\rput(14.7,2.5){\smsize{$2$}}
\psline[linewidth=.04](12,0)(15,-2)\pscircle*(15,-2){.2}
\rput(13.6,-.5){\smsize{$1$}}\rput(14.7,-2.5){\smsize{$4$}}
\psline[linewidth=.04](18,4)(15,2)\pscircle*(18,4){.2}
\rput(16.4,3.5){\smsize{$1$}}\rput(18.3,4.5){\smsize{$3$}}
\psline[linewidth=.04](18,0)(15,2)\pscircle*(18,0){.2}
\rput(16.6,1.5){\smsize{$1$}}\rput(18.3,.6){\smsize{$3$}}
\psline[linewidth=.04](18.5,-2)(15,-2)\pscircle*(18.5,-2){.2}
\rput(16.7,-1.5){\smsize{$1$}}\rput(18.8,-2.5){\smsize{$3$}}
\pscircle*(30,0){.3}\rput(30,-.7){\smsize{$2$}}
\psline[linewidth=.04](30,0)(27,2)\pscircle*(27,2){.2}
\rput(28.6,1.5){\smsize{$2$}}\rput(27.3,2.5){\smsize{$1$}}
\psline[linewidth=.04](30,0)(27,-2)\pscircle*(27,-2){.2}
\rput(28.4,-.5){\smsize{$1$}}\rput(27.3,-2.5){\smsize{$1$}}
\psline[linewidth=.04](23.5,2)(27,2)\pscircle*(23.5,2){.2}
\rput(25.3,2.5){\smsize{$2$}}\rput(23.1,2.5){\smsize{$4$}}
\psline[linewidth=.04](23.5,-2)(27,-2)\pscircle*(23.5,-2){.2}
\rput(25.3,-1.5){\smsize{$3$}}\rput(23.1,-2.5){\smsize{$3$}}
\psline[linewidth=.04](23.5,-2)(23.5,2)\rput(23,0){\smsize{$1$}}
\psline[linewidth=.04](20.5,0)(23.5,-2)\pscircle*(20.5,0){.2}
\rput(21.4,-1.2){\smsize{$2$}}\rput(20,0){\smsize{$1$}}
\psline[linewidth=.04](20.5,0)(20.5,3.5)\pscircle*(20.5,3.5){.2}
\rput(20,1.7){\smsize{$2$}}\rput(20,3.5){\smsize{$2$}}
\psline[linewidth=.04](30,0)(33,2)\pscircle*(33,2){.2}
\rput(31.4,1.5){\smsize{$2$}}\rput(32.7,2.5){\smsize{$1$}}
\psline[linewidth=.04](30,0)(33,-2)\pscircle*(33,-2){.2}
\rput(31.6,-.5){\smsize{$1$}}\rput(32.7,-2.5){\smsize{$4$}}
\psline[linewidth=.04](36,4)(33,2)\pscircle*(36,4){.2}
\rput(34.4,3.5){\smsize{$1$}}\rput(36.3,4.5){\smsize{$3$}}
\psline[linewidth=.04](36,0)(33,2)\pscircle*(36,0){.2}
\rput(34.6,1.5){\smsize{$1$}}\rput(36.3,.6){\smsize{$3$}}
\psline[linewidth=.04](36.5,-2)(33,-2)\pscircle*(36.5,-2){.2}
\rput(34.7,-1.5){\smsize{$1$}}\rput(36.8,-2.5){\smsize{$3$}}
\psline[linewidth=.04](38.5,0)(36.5,-2)\rput(38.4,.4){\smsize{$\bf 1$}}
\end{pspicture}
\caption{Decorated graphs of types $A_3$ and $\ti{A}_{33}$}
\label{Afig}
\end{figure}
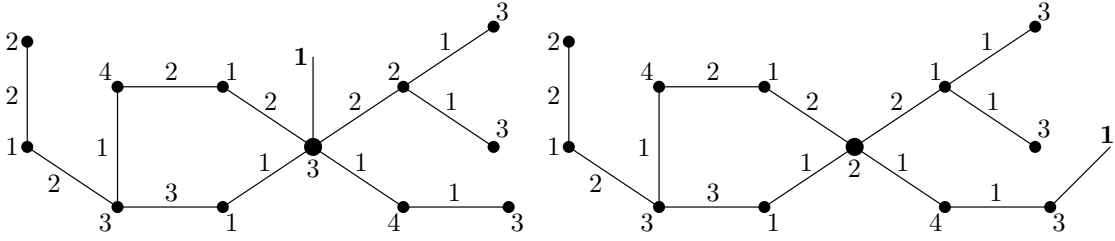

\begin{figure}
\begin{pspicture}(0,-1.3)(10,2)
\psset{unit=.4cm}
\pscircle*(10,0){.3}\rput(10,.7){\smsize{$2$}}
\psline[linewidth=.1](10,0)(7,2)\pscircle*(7,2){.2}
\rput(8.6,1.5){\smsize{$2$}}\rput(7.3,2.5){\smsize{$3$}}
\psline[linewidth=.1](10,0)(6,0)\pscircle*(6,0){.2}
\rput(7.7,.4){\smsize{$2$}}\rput(6.3,.6){\smsize{$3$}}
\psline[linewidth=.04](6,0)(3,2)\pscircle*(3,2){.2}
\rput(4.6,1.4){\smsize{$3$}}\rput(3.4,2.5){\smsize{$2$}}
\psline[linewidth=.04](6,0)(3,-2)\pscircle*(3,-2){.2}
\rput(4.1,-.8){\smsize{$1$}}\rput(3.2,-1.3){\smsize{$1$}}
\psline[linewidth=.1](10,0)(7,-1.5)\pscircle*(7,-1.5){.2}
\rput(7.7,-.7){\smsize{$2$}}\rput(6.8,-.9){\smsize{$3$}}
\psline[linewidth=.04](10,0)(8,-3)\pscircle*(8,-3){.2}
\rput(8.5,-1.6){\smsize{$2$}}\rput(7.5,-3.3){\smsize{$1$}}
\psline[linewidth=.04](10,0)(12,-2)\pscircle*(12,-2){.2}
\rput(11,-1.5){\smsize{$3$}}\rput(12.5,-2.2){\smsize{$3$}}
\psline[linewidth=.05,linestyle=dashed](10,0)(14,0)\pscircle*(14,0){.2}
\psline[linewidth=.04](14,0)(17,-1)\pscircle*(17,-1){.2}
\rput(16.2,-.3){\smsize{$2$}}\rput(17.5,-1){\smsize{$1$}}
\psline[linewidth=.04](14,0)(17,1)\pscircle*(17,1){.2}
\rput(15.6,.9){\smsize{$3$}}\rput(17.2,1.6){\smsize{$3$}}
\psline[linewidth=.04](17,1)(20,2)\pscircle*(20,2){.2}
\rput(18.6,1.9){\smsize{$1$}}\rput(20.2,2.6){\smsize{$1$}}
\psline[linewidth=.04](10,0)(13,2)\pscircle*(13,2){.2}
\rput(11.3,1.4){\smsize{$1$}}\rput(13.2,2.7){\smsize{$3$}}
\psline[linewidth=.04](10,0)(10,-2.7)\rput(10,-3.2){\smsize{$\bf 1$}}
\pscircle*(29,0){.3}\rput(29,.7){\smsize{$2$}}
\psline[linewidth=.1](29,0)(26,2)\pscircle*(26,2){.2}
\rput(27.6,1.5){\smsize{$2$}}\rput(26.3,2.5){\smsize{$3$}}
\psline[linewidth=.1](29,0)(25,0)\pscircle*(25,0){.2}
\rput(26.7,.4){\smsize{$2$}}\rput(25.3,.6){\smsize{$3$}}
\psline[linewidth=.04](25,0)(22,2)\pscircle*(22,2){.2}
\rput(23.6,1.4){\smsize{$3$}}\rput(22.4,2.5){\smsize{$2$}}
\psline[linewidth=.04](25,0)(22,-2)\pscircle*(22,-2){.2}
\rput(23.1,-.8){\smsize{$1$}}\rput(22.2,-1.3){\smsize{$1$}}
\psline[linewidth=.1](29,0)(26,-1.5)\pscircle*(26,-1.5){.2}
\rput(26.7,-.7){\smsize{$2$}}\rput(25.8,-.9){\smsize{$3$}}
\psline[linewidth=.04](29,0)(27,-3)\pscircle*(27,-3){.2}
\rput(27.5,-1.6){\smsize{$2$}}\rput(26.5,-3.3){\smsize{$1$}}
\psline[linewidth=.04](29,0)(31,-2)\pscircle*(31,-2){.2}
\rput(30,-1.5){\smsize{$3$}}\rput(31.5,-2.2){\smsize{$3$}}
\psline[linewidth=.05,linestyle=dashed](29,0)(33,0)\pscircle*(33,0){.2}
\psline[linewidth=.04](33,0)(36,-1)\pscircle*(36,-1){.2}
\rput(35.2,-.3){\smsize{$2$}}\rput(36.5,-1){\smsize{$1$}}
\psline[linewidth=.04](33,0)(36,1)\pscircle*(36,1){.2}
\rput(34.6,.9){\smsize{$3$}}\rput(36.2,1.6){\smsize{$3$}}
\psline[linewidth=.04](36,1)(39,2)\pscircle*(39,2){.2}
\rput(37.6,1.9){\smsize{$1$}}\rput(39.2,2.6){\smsize{$1$}}
\psline[linewidth=.05,linestyle=dashed](29,0)(32,2)\pscircle*(32,2){.2}
\psline[linewidth=.04](32,2)(35,3)\pscircle*(35,3){.2}
\rput(33.6,3){\smsize{$1$}}\rput(35.2,3.6){\smsize{$3$}}
\psline[linewidth=.04](32,2)(33,4)\rput(33.3,4.4){\smsize{$\bf 1$}}
\end{pspicture}
\caption{Decorated rooted trees of types $B_2$ and $\ti{B}_{22}$}
\label{Bfig}
\end{figure}
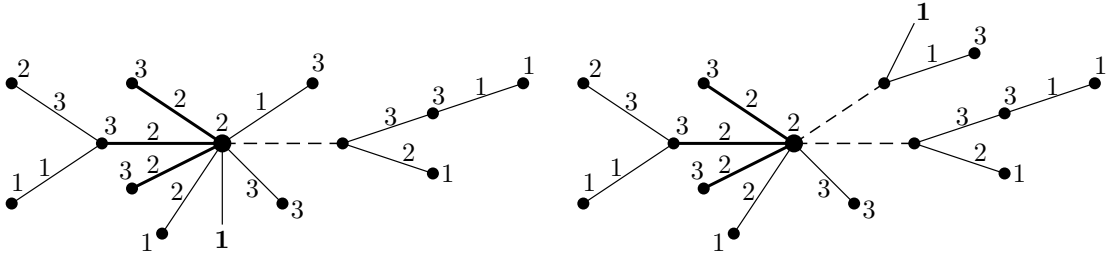

Within each of the 2~types, there are 2~sub-types of graphs, 
depending on whether the marked point is attached to 
the distinguished vertex or some other vertex.
In the former case, a graph has one special vertex label:
the number decorating the vertex to which the marked point is attached. 
In the latter case, a graph has two special vertex labels:
the number decorating the vertex to which the marked point is attached
and the number decorating the distinguished vertex.
Since $\phi_i|_{P_j}\!=\!\de_{ij}$,
only the graphs that describe stable maps taking the marked point to~$P_i$
contribute to~\e_ref{Frestr_e}; 
in these graphs the first special vertex label is~$i$.
Thus, the types of graphs that contribute to~\e_ref{Frestr_e}  
can be described as $A_i$, $\ti{A}_{ij}$, $B_i$, and~$\ti{B}_{ij}$,
with the first subscript describing the label of the vertex to which 
the marked point is attached and the second describing the label of 
the distinguished vertex if this vertex is different from the first
(the label may still be the same).

The approach of \cite{bcov1} to computing the total contribution
to~\e_ref{Frestr_e} of all graphs of a fixed type is to break every graph
at the distinguished vertex, adding  a marked point 
to each of the resulting ``strands" so that each graph is completely
encoded by its strands.
In the case of $B_i$-graphs, all strands are graphs with one marked point.
In the case of $A_i$ and $B_{ij}$-graphs, there is precisely one strand 
with two marked points;
in the former case it contributes to~$\wt\cZ_{ii}^*$,
while in the latter it contributes to~$\cZ_{ji}^*$.
In the case of $A_{ij}$-graphs, there are two strands  with two marked points,
one of which contributes to $\wt\cZ_{jj}^*$,
while the other to~$\cZ_{ji}^*$.
Each of the one-pointed strands contributes to~$\cZ_j^*$.
While the number of one-pointed strands can be arbitrary large,
it is possible to sum up over all arrangements of such strands
because of a special property of the power series~$\cZ_i^*$
described in \cite[Section~2.2]{bcov1}.
This reduces the total contribution, 
$\cA_i$, $\ti{\cA}_{ij}$, $\cB_i$, or $\ti{\cB}_{ij}$ 
of all graphs of a fixed type, 
$A_i$, $\ti{A}_{ij}$, $B_i$, or $\ti{B}_{ij}$, 
to a fairly simple expression involving $\cZ_i^*$, $\cZ_{ij}^*$,
and/or~$\wt\cZ_{ij}^*$.

\begin{prp}\label{equivred_prp}
For every $i\!=\!1,2,\ldots,n$,
\BE{equivred_e}\X(\al,\al_i,Q)
=\cA_i(Q)+\sum_{j=1}^n\ti\cA_{ij}(Q)+\cB_i(Q)+\sum_{j=1}^n\ti\cB_{ij}(Q),\EE
where
\begin{equation*}\begin{split}
\cA_i(Q)&=\frac{1}{\Phi_0(\al_i,Q)}\fR_{\hb_1=0}\left\{\fR_{\hb_2=0}
\left\{\ne^{-\eta_i(Q)/\hb_1}\ne^{-\eta_i(Q)/\hb_2}\wt\cZ_{ii}^*(\hb_1,\hb_2,Q)
 \right\}\right\};\\
\ti\cA_{ij}(Q)&=\frac{\cA_j(Q)}{\prod\limits_{k\neq j}(\al_j\!-\!\al_k)}
\fR_{\hb=0}\left\{e^{-\eta_j(Q)/\hb}\cZ_{ji}^*(\hb,Q)\right\};\\
\cB_i(Q)&=\frac{\lr{\a}\al_i^l}{24} \fR_{\hb=0,\i,-\a\al_i}
\left\{\frac{\prod\limits_{k=1}^n(\al_i\!-\!\al_k\!+\!\hb)}{\hb^3\prod\limits_{r=1}^l(a_r\al_i\!+\!\hb)}
\frac{\cZ_i^*(\hb,Q)}{1\!+\!\cZ_i^*(\hb,Q)}\right\};\\
\ti\cB_{ij}(Q)&=-\frac{1}{\prod\limits_{k\neq j}(\al_j\!-\!\al_k)}
\frac{\lr{\a}\al_j^l}{24} \fR_{\hb=0,\i,-\a\al_j}\left\{
\frac{\prod\limits_{k=1}^n(\al_j\!-\!\al_k\!+\!\hb)}
{\hb^2\prod\limits_{r=1}^l(a_r\al_j\!+\!\hb)}
\frac{\cZ_{ji}^*(\hb,Q)}{1\!+\!\cZ_j^*(\hb,Q)}\right\}.
\end{split}\end{equation*}
\end{prp}

This proposition is essentially proved in \cite[Sections~1.3,~1.4,~2]{bcov1},
which treats the $l\!=\!1$ case.
In the general case, the $\T$-fixed loci and their normal bundles 
remain the same.
The only required changes involve the Euler classes of
the bundles $\V_0'$ and $\V_1$, which are now products of
the Euler classes of the bundles in~\cite{bcov1}.
These changes are:
\begin{enumerate}
\item equation \cite[(1.37)]{bcov1} becomes
$$\E(\V_1)|_{\wt\cZ_{\Ga}}
=\lr{\a}\al_{\mu(v_0)}^l\prod_{e\in\Edg(v_0)}\!\!\!\!\!\!\!\pi_e^*\E(\V_0')
\Big/\prod\limits_{r=1}^l\big(a_r\al_{\mu(v_0)}\!+\!\psi_{\Ga}\!+\!\la\big);$$
\item $n\al_i+\hb$ is replaced by $\prod\limits_{r=1}^l(a_r\al_i+\hb)$ 
in the definition of $\Psi$ above \cite[(2.19)]{bcov1},
in \cite[(2.23) and (2.24)]{bcov1},  and 
in the last equation in \cite[Section~2.3]{bcov1},
leading to the corresponding modification in the final expressions for~$\cB_i$
and~$\ti\cB_{ij}$ above;
\item $n\al_i$ is replaced by $\lr{\a}\al_i^l$ in \cite[(2.24)]{bcov1} 
and $n\al_j$ is replaced by $\lr{\a}\al_j^l$
in the last equation in \cite[Section 2.3]{bcov1},
leading to the corresponding modification in the final expressions for~$\cB_i$
and~$\ti\cB_{ij}$ above.
\end{enumerate}

\section{Some properties of hypergeometric series $\ti\F$}
\label{hg_sec}

In this section we study properties of the hypergeometric series $\ti\F$
of \e_ref{F0def} which are  used in Section~\ref{mainthmpf_subs} to 
deduce Theorem~\ref{main_thm} from Proposition~\ref{equivred_prp}.
The results in this section extend most of the statements 
and proofs in~\cite{ZaZ}, which treats the $l\!=\!1$ case.

Let $\bM\!:\cP\!\lra\!\cP$ be as in \e_ref{Mdfn} and define $\F\!\in\!\cP$ by
\begin{equation}\label{F0def_e} 
\F(w,q)\equiv \sum_{d=0}^{\i}q^d
\frac{\prod\limits_{k=1}^l\prod\limits_{r=1}^{a_kd}(a_kw\!+\!r)}
{\prod\limits_{r=1}^d\left[(w\!+\!r)^n-w^n\right]}\,.
\end{equation}
By \e_ref{F0def} and \e_ref{Idfn}
\begin{equation}\label{Iprp_e}
I_p(q)\equiv \bM^p\ti\F(0,q) =\bM^p\F(0,q)
\qquad\forall\,p=0,1,\ldots,n\!-\!1.
\end{equation}
Some advantages of the power series $\F$ over $\ti\F$ are illustrated by 
Lemmas~\ref{HGper_lmm} and~\ref{reg_lmm} below.

\begin{lmm}\label{HGper_lmm} 
The hypergeometric series $\F$ satisfies $\bM^n\F=\F$. 
\end{lmm}

\begin{lmm}[{\cite[Lemma 1.3 and its proof]{ZaZ}}]\label{reg_lmm} 
If $F\!\in\!\cP$ and $\bM^kF=F$ for some $k>0$,  then every coefficient 
of the power series $\,\log \bM^pF(w,q)\in\Q(w)\big[\big[q\big]\big]$ is 
$\O(w)$ as $w\to\i\,$ for all $p\geq 0$. 
Moreover, $\fR_{w=0}\left\{\log\bM^pF(w^{-1},q)\right\}$ does not depend on $p$. 
\end{lmm}

Applying this lemma to $F=\F$, we find that $\F(w,q)$ has  an asymptotic expansion 
\BE{asym} \F(w,q)\;\sim\;\ne^{\mu(q)w}\sum_{s=0}^{\i}\Phi_s(q)\,w^{-s} 
\qquad(w\to\i) \EE
for some power series 
$\mu,\Phi_0,\Phi_1,\dots$ in $\Q\big[\big[q\big]\big]$.
Let 
\BE{Ldfn_e0} L(q)\equiv(1-\a^{\a}q)^{-1/n}\in\Q\big[\big[q\big]\big].\EE

\begin{prp}\label{HG_thm2}  
The power series $\mu$, $\Phi_0$, and $\Phi_1$ in~$\e_ref{asym}$ are given~by
\begin{gather}\label{muPhi0_e} 
\mu(q)=\int_0^q\!\frac{L(u)-1}u\,\nd u\,, \qquad 
\Phi_0=L^{\frac{l+1}{2}}\,, \\
\label{Phi1_e}\begin{split}
\Phi_1=\left\{ \left[\frac{n\!-\!1}{24}-\frac{1}{12}\sum_{r=1}^l\frac{1}{a_r}
\right](L\!-\!1)
-\frac{(n\!-\!2)(n\!+\!1)-3(l^2\!-\!1)}{24n}(L^n\!-\!1)\right\}\frac{\Phi_0}{L}\,.&
\end{split}
\end{gather}
\end{prp}


The last proposition of this section concerns properties of $\ti\F$ and $\F$
around $w\!=\!0$.

\begin{prp}\label{HG_thm1} 
The power series $I_p(q)$ defined by \e_ref{Idfn} satisfy
\begin{alignat}{1}
\label{HG1e0}  & I_{p}(q)=1 \qquad\forall~  p=n\!-\!l\!+\!1,\ldots,n\!-\!1,\\
\label{HG1e3} & I_p(q)=I_{n-l-p}(q)\qquad \forall~p=0,1,\ldots,n\!-\!l,\\
\label{HG1e1}  & I_0(q)\,I_1(q)\,\ldots\,I_{n-l}(q)=\big(1-\a^{\a}q\big)\iv,\\
\label{HG1e2}  &  I_0(q)^{n-l}I_1(q)^{n-l-1}\ldots I_{n-l-1}(q)^1 I_{n-l}(q)^0
                            =\big(1-\a^{\a}q\big)^{-(n-l)/2}\,.
\end{alignat}\end{prp}

While \e_ref{HG1e3} and \e_ref{HG1e1} imply \e_ref{HG1e2},
\e_ref{HG1e2} is simpler to prove directly than~\e_ref{HG1e3} 
and will be verified together with~\e_ref{HG1e1}, as is done in~\cite{ZaZ}.

Lemma~\ref{HGper_lmm} and Propositions~\ref{HG_thm2} and~\ref{HG_thm1} 
are proved in Sections~\ref{per_subs}-\ref{Phi1_subs} following
the approach in~\cite{ZaZ}.
Let  
$$D=q\,\frac \nd{\nd q}\,,D_w=D\!+\!w\!: 
\Q(w)\big[\big[q\big]\big]\lra \Q(w)\big[\big[q\big]\big].$$
Thus,
\begin{equation}\label{Dw_e}\begin{split}
D_w\left[\sum_{d=0}^{\i} c_d(w)q^d\right]&=
\sum_{d=0}^{\i} (w\!+\!d)c_d(w)q^d\,,\\
\bM F(w,q)&= w\iv D_w\bigl[F(w,q)/F(0,q)\bigr]
\qquad\forall\,F\!\in\!\cP.
\end{split}\end{equation}

\subsection{Proof of Lemma~\ref{HGper_lmm} and Proposition~\ref{HG_thm1}}  
\label{per_subs}

We will repeatedly use the following lemma.

\begin{lmm}[{\cite[Corollary 2.2]{ZaZ}}]\label{ode_cor}  
Suppose $F(w,q)\in\cP$ satisfies
\begin{equation}\label{ode_cor_e1} 
\biggl(\sum_{r=0}^mC_r(q)\,D_w^{\,r}\biggr)F(w,q)\=A(w,q) 
\end{equation}
for some power series $C_0(q),\,\dots,\,C_m(q)\in\Q\big[\big[q\big]\big]$ 
and $A(w,q)\in\Q(w)\big[\big[q\big]\big]$ with $A(0,q)\equiv0$.  
Then 
\begin{equation}
\label{ode_cor_e2}
 \biggl(\sum_{s=0}^{m-1}\ti C_s(q)\,D_w^{\,s}\biggr)\bM F(w,q)\=\frac1w\,A(w,q)\,, 
 \end{equation}
where $\ti C_s(q)\equiv\sum\limits_{r=s+1}^m\binom{r}{s+1}\,C_r(q)\,D^{r-1-s}F(0,q)\,$.   \end{lmm}

Define power series $\F_{-l},\F_{-l+1},\ldots\in\cP$ by 
\begin{equation}\label{Ifuncdfn_e1a}\begin{split} 
\F_{-l}(w,q)&\equiv  \sum_{d=0}^{\i}q^d 
\frac{\prod\limits_{k=1}^l\prod\limits_{r=0}^{a_kd-1}(a_kw\!+\!r)}
  {\prod\limits_{r=1}^d\left[(w\!+\!r)^n\!-\!w^n\right]}\,,\\
\F_p&\equiv\bM^{l+p}\F_{-l}\qquad \forall\, p>-l. 
\end{split}\end{equation}
Using \e_ref{Dw_e}, we find that 
\BE{Fpprop_e}\F_p(0,q)=1 ~~~\forall\,p=-l,-l\!+\!1,\ldots,-1, \quad 
w^{-l}D_w\F_{-1}=\F;\EE
thus, $\F_p=\bM^p\F$ for all $p\ge0$. 

It is also straightforward to check that $\F_{-l}$ solves the differential equation
\BE{HGODE_e1a} 
\left\{D_w^n\,-\,q\,\prod_{k=1}^l\,
\prod_{r=0}^{a_k-1}\left(a_kD_w+r\right)\right\}\,\F_{-l} \=w^n\,\F_{-l}\,.\EE
This equation is of the form \e_ref{ode_cor_e1} with $F=\F_{-l}$, 
$A=w^n\F_{-l}$, $m=n$, 
\begin{equation}\label{init_e} 
 C_n(q)=1-\a^{\a}q\,,\qquad C_{n-1}(q)=-\frac{n-l}{2}\a^{\a}q.
\end{equation}
Applying Lemma~\ref{ode_cor} repeatedly, we obtain
\BE{ode-p} 
\sum_{s=0}^{n-l-p}C_s^{(p)}(q)\,D_w^s\F_p(w,q)\=w^{n-l-p}\,\F_{-l}(w,q)\,,
\qquad -l\!\le p\le n\!-\!l, \EE
where by the first identity in \e_ref{Fpprop_e} and by \e_ref{Iprp_e}
\begin{equation*}\begin{split}
C_s^{(p)}(q)&=C_{s+l+p}^{(-l)}(q)  \qquad -l\le p\le0,\\ 
C_s^{(p)}(q)&=\sum_{r=s+1}^{n-l-p+1}
 \binom{r}{s\!+\!1}\,C_r^{(p-1)}(q)\,D^{r-s-1}I_{p-1}(q) \qquad p>0.
\end{split}\end{equation*}
Using \e_ref{init_e} and induction on $p$, we find that the top two coefficients 
in \e_ref{ode-p} are given~by
\begin{alignat}{1}
\label{coeff1} 
C^{(p)}_{n-l-p} &= \left(1-\a^{\a}q\right)\,\prod_{r=0}^{p-1}I_r(q)\,,\\
\label{coeff2} 
C^{(p)}_{n-l-p-1} &= \left[-\a^{\a}q\frac{n-l}{2}
\,+\,\left(1-\a^{\a}q\right)\,\sum_{r=0}^{p-1}(n\!-\!l\!-\!r)\,
\frac{DI_r(q)}{I_r(q)}\right]\,\prod_{r=0}^{p-1}I_r(q)\,.
\end{alignat}

Setting $p=n\!-\!l$ in \e_ref{ode-p} and \e_ref{coeff1} thus gives
\BE{conseq1_e} 
\big(1-\a^{\a}q\big)\,
\left(\prod_{r=0}^{n-1-l}I_r(q)\right)\F_{n-l}(w,q)\=\F_{-l}(w,q)\,.\EE
Setting $w=0$ in~\e_ref{conseq1_e} and using $\F_{-l}(0,q)=1$ gives~\e_ref{HG1e1}. 
Substituting~\e_ref{HG1e1}
back into \e_ref{conseq1_e} gives $\F_{n-l}/I_{n-l}=\F_{-l}$ and thus
$\F_{n-l+1}\!=\!\F_{-l+1}$.
Applying $\bM$ to both sides of the last identity $l\!-\!1$ times
and using~\e_ref{Fpprop_e}, we obtain Lemma~\ref{HGper_lmm} and equation~\e_ref{HG1e0}.
Similarly, setting $p=n\!-\!l\!-\!1$ in \e_ref{ode-p}-\e_ref{coeff2}  
and then taking $w=0$ gives 
$$ \sum_{r=0}^{n-1-l} (n\!-\!l\!-\!r)\,\frac{DI_r(q)}{I_r(q)}
=\frac{n\!-\!l}{2}\,\frac{\a^{\a}q}{1-\a^{\a}q}\,.$$
Integrating this identity and then exponentiating, we obtain \e_ref{HG1e2}.

We next prove the reflection symmetry~\e_ref{HG1e3}.
The function $\hF\!\in\!\cP$ defined in \e_ref{F0def}
satisfies the differential equation
$$\left\{D_w^{n-l}
-\lr{\a}q\prod_{k=1}^l\prod_{r=1}^{a_k-1}\bigl(a_kD_w+r\bigr)\right\}\,\hF
=w^{n-l}.$$
This equation is of the form \e_ref{ode_cor_e1} with $F=\ti\F$, 
$A=w^{n-l}$, $m=n\!-\!l$, and 
$$C_{n-l}(q)=1-\a^{\a}q.$$
Applying Lemma~\ref{ode_cor} repeatedly, we obtain
\BE{Fpde_e} \sum_{s=0}^{n-l-p}\ti{C}_s^{(p)}(q)D_w^s\bM^p\hF(w,q)
=w^{n-l-p}\,,\qquad 0\le p\le n\!-\!l,\EE
with $\ti{C}_{n-l-p}^{(p)}\!=\!C_{n-l-p}^{(p)}$ given by \e_ref{coeff1}.
Setting $p\!=\!n\!-\!l$ in \e_ref{Fpde_e}
and using~\e_ref{coeff1} and~\e_ref{HG1e1}, we find that 
$$\bM^{n-l}\hF(w,q)=I_{n-l}(q)$$ 
is independent of~$w$.
Using~\e_ref{Dw_e} and downward induction on~$p$, we then find that
\BE{hFexp_e} w^{l-n}\,\hF(w,q)
=I_0\,D_w\iv\,I_1\,D_w\iv\,\ldots I_{n-l-1}\,D_w\iv\, I_{n-l}, \EE
where
$$D_w^{-1}\left[\sum_{d=0}^{\i} c_d(w)q^d\right]=
\sum_{d=0}^{\i} \frac{c_d(w)}{(w\!+\!d)}q^d\,.$$
Comparing the coefficients of $q^d$ on the two sides of~\e_ref{hFexp_e}, 
we find that
$$\frac{\lr{\a}\iv
\prod\limits_{k=1}^l\prod\limits_{r=0}^{a_kd}(a_kw\!+\!r)}{[w(w\!+\!1)\ldots(w\!+\!d)]^n}
=\sum_{\begin{subarray}{c}d_0+\ldots+d_{n-l}=d\\ d_0,\ldots,d_{n-l}\ge0\end{subarray}}
\frac{c_0(d_0)\ldots c_{n-l}(d_{n-l})}
{(w\!+\!d_1\!+\!\ldots\!+\!d_{n-l})(w\!+\!d_2\!+\!\ldots\!+\!d_{n-l})\ldots(w\!+\!d_{n-l})}$$
for all $d\ge0$, where $c_p(d)$ is the coefficient of $q^d$ in $I_p(q)$. 
This identity is equivalent~to
\BE{Fcoeff_e}\sum_{p=0}^{n-l}\frac{c_p(d)}{w^{n-l-p}(w\!+\!d)^p}
=\frac{\prod\limits_{k=1}^l\prod\limits_{r=0}^{a_kd}(a_kw\!+\!r)}
     {\lr{\a}\prod\limits_{r=0}^d(w\!+\!r)^n} - 
\sum_{\begin{subarray}{c}d_0+\ldots+d_{n-l}=d\\ 0\le d_0,\ldots,d_{n-l}<d\end{subarray}}
\frac{c_0(d_0)\ldots c_{n-l}(d_{n-l})}
   {(w\!+\!d_1\!+\!\ldots\!+\!d_{n-l})\ldots(w\!+\!d_{n-l})}.\EE
We will use this identity to show by induction that 
\BE{Find_e} c_p(d)=c_{n-l-p}(d)  \qquad\forall~p=0,1,\ldots,n\!-\!l\,,\EE
thus establishing~\e_ref{HG1e3}.
Since $c_p(0)\!=\!I_p(0)\!=\!1$ for all $p$, \e_ref{Find_e} holds for $d\!=\!0$.
Suppose $d\!\ge\!1$ and \e_ref{Find_e} holds with $d$ replaced by every $d'\!<\!d$. 
The substitution $w\to-w\!-\!d$ acts by $(-1)^{n-l}$ on the first term 
on the right-hand side of~\e_ref{Fcoeff_e}.
It acts in the same way on the second term by the induction assumption; 
this can be seen by the renumbering 
$$\big(d_0,\ldots,d_{n-l}\big)\lra \big(d_{n-l},\ldots,d_0\big).$$
Thus, the substitution $w\to-w\!-\!d$ acts by $(-1)^{n-l}$ on the left-hand side
of~\e_ref{Fcoeff_e}, and so $c_p(d)=c_{n-l-p}(d)$ for all $0\le p\le n\!-\!l$,
as needed for  the inductive step.

\subsection{Proof of \e_ref{muPhi0_e}}
\label{muPhi0_subs}

By Lemmas~\ref{HGper_lmm} and~\ref{reg_lmm}, the functions $\F_p(w,q)\equiv\bM^p\F(w,q)$
admit asymptotic expansions
\BE{asym_p} 
\F_p(w,q)\sim\ne^{\mu(q)w}\sum_{s=0}^{\i}\Phi_{p,s}(q)\,w^{-s} \qquad(w\to\i), \EE
with the same function $\mu(q)$ in the exponent for all $p$.
Since $\F_{0}\!=\!\F$ and $\F_{p+1}\!=\!\bM \F_p$,
\BE{expind_e}   \Phi_{0,s}=\Phi_s\,, \qquad
\Phi_{p+1,s}=\frac{1+D\mu}{I_p}\,\Phi_{p,s}\,+\, 
\begin{cases} D\left(\frac{\Phi_{p,s-1}}{I_p}\right),&\hbox{if}~s\ge1,\\ 
 0,&\hbox{if}~s=0. \end{cases}\EE

Taking $s\!=\!0$ in \e_ref{expind_e}, we find by induction that 
\BE{Phi0_e}\Phi_{p,0}=\frac{(1+D\mu)^p}{I_0I_1\ldots I_{p-1}}\,\Phi_0\,.\EE
Since $\F_n\!=\!\F_0$ by Lemma~\ref{HGper_lmm} and $\Phi_0(0)\!=\!1$, 
setting $p\!=\!n$ in the above identity we obtain
$$(1+D\mu)^n=I_0\ldots I_{n-1}\,.$$
The first claim in \e_ref{muPhi0_e} now follows from 
\e_ref{HG1e0} and \e_ref{HG1e1}.

For each $p\ge0$, let
$$H_p(q)\equiv\frac{L^p(q)}{I_0(q)\ldots I_{p-1}(q)}.$$
By definition, \e_ref{HG1e0}, \e_ref{HG1e1}, \e_ref{HG1e2},
\e_ref{Phi0_e}, and the first identity in \e_ref{muPhi0_e},
\BE{HpProperties}  H_0=H_n=1,\qquad  
\quad H_1H_2\ldots H_n=L^{-\frac{n(l-1)}{2}}, \qquad
\Phi_{p,0}= H_p\Phi_0\,.\EE
Taking $s\!=\!1$ in \e_ref{expind_e} and using the first and last equations above,
we find inductively that
$$\Phi_{p,1}=H_p\left(\Phi_1+p\frac{D\Phi_0}{L}-
 p\frac{\Phi_0DL}{L^2}+\frac{\Phi_0}{L}\sum_{r=1}^p\frac{DH_r}{H_r}\right)
\qquad\forall p\,\ge0.$$  
Setting $p\!=\!n$ in this relation and using $\Phi_{n,1}\!=\!\Phi_1$,
along with the first and second equations in~\e_ref{HpProperties}, 
we find that 
$$\frac{D\Phi_0}{\Phi_0}=\frac{l\!+\!1}{2}\,\frac{DL}{L}\,.$$
Since $\Phi_0(0)\!=\!1\!=\!L(0)$, this confirms the second claim 
in~\e_ref{muPhi0_e}.

\subsection{Proof of \e_ref{Phi1_e}}
\label{Phi1_subs}

The argument in Section~\ref{muPhi0_subs} can be systematized as in \cite{ZaZ}
to obtain an algorithm for computing  every $\Phi_s$ by a differential recursion. 

Define $\xi_s\!\in\!\Q$ by
$$\prod_{k=1}^l\prod_{j=1}^{a_k}\left(a_kD+j\right)\equiv
  \a^{\a}\sum_{s=0}^n\xi_sD^s\in\Z[D];$$
thus, $\xi_n=1$, $\xi_{n-1}=(n\!+\!l)/2$, and
\BE{xi_e}\begin{split} 
\xi_{n-2} &=\frac{1}{24}\sum_{k=1}^l \frac{(a_k-1)(a_k+1)(3a_k+2)}{a_k}
+\frac{1}{4}\sum_{1\leq i<j\leq l}(1+a_i)(1+a_j)\\
&=-\frac{1}{12}\sum_{r=1}^l\frac{1}{a_r}+\frac{3n^2+n(6l-4)+3l^2-6l}{24}.
\end{split}\EE
Let 
$$\hD\equiv D\!+\!Lw\!:\Q(w)\big[\big[q\big]\big]\lra\Q(w)\big[\big[q\big]\big].$$ 
The series  $\oF(w,q)\equiv\ne^{-\mu(q)w}\F(w,q)$ admits an asymptotic expansion
\BE{barFexp_e}\oF(w,q)\sim \sum_{s=0}^{\i}\Phi_s(q)\,w^{-s}  \qquad(w\to\i).\EE
Since $1\!+\!D\mu\!=\!L$ by the first claim in~\e_ref{muPhi0_e}
and $\F(w,q)$ satisfies the ODE 
$$\left\{D_w^n-w^n-q\prod_{k=1}^l\prod_{r=1}^{a_k}\left(a_kD_w+r\right)\right\}\F=0\,,$$
the series  $\oF(w,q)$ satisfies the ODE
\BE{barFODE_e} \fL\oF=0,\EE 
where $\fL$ is the differential operator
\BE{Lexp_e}\begin{split}
\fL \equiv L^n\left[\hD^n-w^n-q\prod_{k=1}^l\prod_{r=1}^{a_k}\big(a_k\hD+r\big)\right] 
= \hD^n-(Lw)^n-(L^n\!-\!1)\sum_{s=0}^{n-1}\xi_s\hD^s\,.
\end{split}\EE

Since $DL\!=\!L(L^n\!-\!1)/n$, as in \cite[Section~2.4]{ZaZ}
\BE{Dexp_e}
\hD^s=\sum_{k=0}^s\sum_{i=0}^k\binom{s}{i}\H_{s-i,k-i}(L^n)(Lw)^{s-k}D^i\,,\EE
where the polynomials $\H_{m,j}\in\Q[X]$ are defined by 
\begin{equation*}\label{Hdfn_e}\begin{split}
\H_{m,j}&\equiv0 \quad\hbox{if}~~m<0,~\hbox{or}~j\!<\!0,~\hbox{or}~j\!>\!m,\qquad
\H_{0,0}\equiv1;\\
\H_{m,j}(X)&\equiv\H_{m-1,j}(X)+
(X\!-\!1)\bigg(X\frac{\nd}{\nd X}+\frac{m\!-\!j}{n}\bigg)\H_{m-1,j-1}(X)
\quad\hbox{if}~ m\ge1,~0\le j\le m. 
\end{split}\end{equation*}
In particular, for $m\!\ge\!0$
\BE{Hlow_e}\begin{split} 
\H_{m,0}(X)&=1, \qquad \H_{m,1}(X)=\frac{1}{n}\binom{m}{2}(X\!-\!1), \\
\H_{m,2}(X)&=\frac{1}{n^2}\binom{m}{3}\big((n\!+\!1)X-1\big)(X\!-\!1)
  +\frac{3}{n^2}\binom{m}{4}(X\!-\!1)^2\,.
\end{split}\EE
By \e_ref{Lexp_e} and \e_ref{Dexp_e}, 
\BE{Lexp_e2} \fL=\sum_{k=1}^n(Lw)^{n-k}\fL_k\,, \EE
where $\fL_k$ is the differential operator of order $k$ given by
\BE{Ldfn_e}
\fL_k=\sum_{i=0}^k\left[\binom{n}{i}\H_{n-i,k-i}(L^n) 
-(L^n\!-\!1)\sum_{r=1}^{k-i}\binom{n\!-\!r}{i}\xi_{n-r}\,\H_{n-i-r,k-i-r}(L^n)\right]D^i\,.
\EE
By \e_ref{Hlow_e}, the first two of these operators are
\BE{L12_e}\begin{split}  
\fL_1&=nD-\frac{l\!+\!1}{2}(L^n\!-\!1)
=nL^{\frac{l+1}{2}}DL^{-\frac{l+1}{2}}\,,\\
\fL_2&=\binom{n}{2}D^2-\frac{(l\!+\!2)(n\!-\!1)}{2}(L^n\!-\!1)D\\
&\qquad+\left[\frac{(n\!-\!1)(n\!-\!2)(n\!-\!6l\!-\!5)}{24n}L^n
   +\frac{(3n\!+\!6l\!+\!5)(n\!-\!1)(n\!-\!2)}{24n}-\xi_{n-2}\right](L^n\!-\!1)\,.
\end{split}\EE
Combining \e_ref{barFexp_e}, \e_ref{barFODE_e}, and \e_ref{Lexp_e2}, 
we obtain the following.          
            
\begin{prp}\label{Phithm} 
The power series $\Phi_s\in\Q\big[\big[q\big]\big]$, $s\ge0$, defined by \e_ref{asym},
are determined by the first-order ODEs
\begin{equation}\label{PhiODE}
  \fL_1(\Phi_s)\+\frac1L\,\fL_2(\Phi_{s-1})\+\frac1{L^2}\fL_3(\Phi_{s-2})
   \+\ldots\+\frac1{L^{n-1}}\,\fL_n(\Phi_{s+1-n})\=0, \quad s\ge0,
   \end{equation}
together with the initial conditions $\Phi_s(0)\!=\!\de_{0,s}$
and $\Phi_s=0$ for $s<0$.
\end{prp}

The $s\!=\!0$ case of \e_ref{PhiODE} immediately recovers 
the second claim in~\e_ref{muPhi0_e}.
The $s\!=\!1$ case of \e_ref{PhiODE} then gives
$$nD\big(\Phi_1/\Phi_0\big)
=-\frac{1}{L}(L^n\!-\!1)\left(
\frac{n^2\!-\!n\!-\!3l^2\!+\!1}{24n}(n\!-\!1)L^n
+\left[\frac{1}{12}\sum_{r=1}^l\frac{1}{a_r}-\frac{3l^2\!-\!1}{24n}\right]\right).$$
Along with $DL\!=\!L(L^n\!-\!1)/n$ and $\Phi_1(0)=0$, this identity gives \e_ref{Phi1_e}.

\section{Computation of reduced genus~1 GW-invariants}
\label{alg_sec}

In this section, we deduce Theorem~\ref{main_thm} below from 
Proposition~\ref{equivred_prp}, using Lemmas~\ref{algebr_lmm1} and~\ref{Z_lmm} 
and the properties of the hypergeometric series~$\F(w,q)$
described by Proposition~\ref{HG_thm2} and~\ref{HG_thm1}.
Lemma~\ref{algebr_lmm1} is used to drop purely equivariant terms from
the power series~$\X$, while Lemma~\ref{Z_lmm} provides 
the relevant information about the genus~0 generating functions $\cZ^*_i$, $\cZ^*_{ji}$, and~$\wt{\cZ}^*_{ii}$.

\begin{mythm}\label{main_thm}
The generating function $\X_0(Q)$ defined by \e_ref{pushclass_e} is given~by
$$\X_0(Q)=Q\frac{\nd}{\nd Q}\left(\ti{A}(q)+\ti{B}(q)\right),$$
where $Q$ and $q$ are related by the mirror map \e_ref{mirmap_e} and
\begin{alignat*}{1}
\ti{A}(q)&= \frac{n}{48}\!\left(n\!-\!1- 2\sum_{k=1}^l\frac{1}{a_k}\right)\mu(q)\\
&\quad -\begin{cases}
\frac{n+1}{48}\log\left(1\!-\!\a^{\a}q\right)+\!\!
\sum\limits_{p=0}^{\frac{n-2-l}{2}}\!\!\frac{(n-l-2p)^2}{8}\log I_p(q),
&\hbox{if}~2|(n\!-\!l);\\
\frac{n-2}{48}\log\left(1\!-\!\a^{\a}q\right)+\!\!
\sum\limits_{p=0}^{\frac{n-3-l}{2}}\!\!\frac{(n-l-2p)^2-1}{8}\log I_p(q),
&\hbox{if}~2\!\not|(n\!-\!l);
\end{cases}\\
\ti{B}(q)&=-\frac{n}{48}\left(n\!-\!1\!-\!2\sum_{k=1}^l\!\frac{1}{a_k}\right)\mu(q)
+\frac{\lr{\a}}{24}\eps_0(\a)\left[\log I_0(q)\right]
+\frac{\lr{\a}}{24}\eps_1(\a)J(q)\\
&\quad+\frac{l\!+\!1}{48}\log\left(1\!-\!\a^{\a}q\right)
+\frac{\lr{\a}}{24}\sum_{p=2}^{n-1-l}\!
\left\llbracket
\frac{(1\!+\!w)^n}{\prod\limits_{k=1}^l(1\!+\!a_k w)}\right\rrbracket_{w;n-1-l-p}
\hspace{-.5in}\left\llbracket\log \ti\F(w,q)\right\rrbracket_{w;p}
,
\end{alignat*}
where $\ti\F$, $I_p$, $J$, $\mu$, and $\lrbr{\cdot}_{w;p}$ 
are defined by~\e_ref{F0def}, \e_ref{Idfn}, \e_ref{mirmap_e}, \e_ref{asym}
and~\e_ref{coeff_e}, respectively.
\end{mythm}

We will show that the terms $\cA_i$ and $\ti\cA_{ij}$, with $j\!\in\![n]$,
in~\e_ref{equivred_e} together contribute $\frac{1}{2}Q\frac{\nd}{\nd Q}\bA(q)$
to $\X_0(Q)$, where
\BE{main_thm_e1}\begin{split}
\bA(q)=\frac{n}{24}\left(n\!-\!1\!-\!2\sum_{r=1}^l\!\frac{1}{a_r}\right)\mu(q)
-\frac{3(n\!-\!1\!-\!l)^2+(n\!-\!2)}{24}\log\left(1-\a^{\a}q\right)
\qquad&\\
-\sum_{p=0}^{n-2-l}\binom{n\!-\!l\!-\!p}{2}\log I_p(q),&
\end{split}\EE
while the terms $\cB_i$ and $\ti\cB_{ij}$, with $j\!\in\![n]$,
together contribute $Q\frac{\nd}{\nd Q}\ti{B}(q)$.
Since
\begin{eqnarray*}
\sum_{r=0}^{\frac{n-2-l}{2}}\log I_r(q)+\frac{1}{2}\log I_{\frac{n-l}{2}}(q)
=-\frac{1}{2}\log\left(1\!-\!\a^{\a}q\right), \qquad
\hbox{if}~2|(n\!-\!l),\\
\sum_{r=0}^{\frac{n-1-l}{2}}\log I_r(q)
=-\frac{1}{2}\log\left(1\!-\!\a^{\a}q\right),
\qquad \hbox{if}~2\!\not|(n\!-\!l),
\end{eqnarray*}
by~\e_ref{HG1e3} and~\e_ref{HG1e1},
the expression on the right-hand side of~\e_ref{main_thm_e1} 
equals twice the right-hand side in the first equation in Theorem~\ref{main_thm}.

\subsection{Some algebraic notation and observations}
\label{symmalg_subs}

This section recalls the statement of \cite[Lemma 3.3]{bcov1},
which shows that  most terms appearing in the computation of $\X(\al,x,Q)$ 
have no effect on~$\X_0(Q)$.
We then set up additional related notation and make a few algebraic observations
that help streamline computations in the remainder of the paper.

For each $p\in[n]$, let $\si_p$ be the $p$-th elementary symmetric 
polynomial in $\al_1,\ldots,\al_n$.
Denote by
$$\Q[\al]^{S_n}\equiv\Q[\al_1,\ldots,\al_n]^{S_n}\subset\Q [\al_1,\ldots,\al_n]$$
the subspace of symmetric polynomials, by $\cI\subset\Q[\al]^{S_n}$ 
the ideal generated by $\si_1,\ldots,\si_{n-1}$, and~by
$$\ti\Q[\al]^{S_n} \equiv 
\Q[\al_1,\ldots,\al_n]_{\lr{\al_j,(\al_j-\al_k)|j\neq k}}^{S_n} \subset \Q_{\al}$$
the subalgebra of symmetric rational functions in $\al_1,\ldots,\al_n$
whose denominators are products of $\al_j$ and $(\al_j\!-\!\al_k)$ with $j\!\neq\!k$.
For each $i\!=\!1,\ldots,n$, let 
$$\ti\Q_i[\al]^{S_{n-1}} \equiv 
\Q[\al_1,\ldots,\al_n]_{\lr{\al_i,(\al_i-\al_k)|k\neq i}}^{S_{n-1}}
\subset \Q_{\al}$$
be the subalgebra consisting of rational functions symmetric in $\{\al_k\!:k\!\neq\!i\}$
and with denominators that are  products of $\al_i$ and 
$(\al_i\!-\!\al_k)$ with $k\!\neq\!i$.
Let
\begin{equation}\label{congspacedfn_e}
\K_i\equiv\Span_{\Q}\big\{\cI\cdot\ti\Q_i[\al]^{S_{n-1}},\al_i^{n-2}\cdot\cI\ti\Q[\al]^{S_n},
\{1,\al_i,\ldots,\al_i^{n-3},\al_i^{n-1}\}\cdot\ti\Q[\al]^{S_n}\big\}.
\end{equation}

\begin{lmm}[{\cite[Lemma 3.3]{bcov1}}\footnotemark]\label{algebr_lmm1}
 If $n\!\ge\!2$, the linear span of $\al_i^{n-2}$ is disjoint from $\K_i$:
$$\Span\big\{\al_i^{n-2}\big\}\cap \K_i=\{0\}\subset \Q_{\al}.$$
\end{lmm}

\footnotetext{The definition of $\K_i$ in \cite{bcov1} is missing
$\al_i^{n-2}\cI\ti\Q[\al]^{S_n}$, but the proof of \cite[Lemma 3.3]{bcov1}
still goes through.
This change adds the term $\al_i^{n-1}g_{n-1}$, with $g_{n-1}\!\in\!\cI$, 
to the second numerator in \cite[(3.13)]{bcov1} and $g_{n-1}$
to the right-hand side of \cite[(3.15)]{bcov1}.
As $g\!\in\!\cI$, this addition has no effect on the concluding sentence
in the proof of Lemma~3.3 in~\cite{bcov1}.}

For each $i\!=\!1,\ldots,n$, let
$$\hQ\equiv
\Q(\hb,\al_i)[\al]_{\lr{(\al_i-\al_k+r\hb)|k\in i,r\in\Z,k\neq i}}^{S_{n-1}}
\subset\Q_{\al}(\hb)$$
be the subalgebra consisting of rational functions symmetric in $\{\al_k\!:k\!\neq\!i\}$
and with denominators that are a product of a polynomial with rational coefficients in $\hb$ and $\al_i$ and
of linear factors of the form $(\al_i\!-\!\al_k\!+\!r\hb),r\in\Z$.
Denote by 
$$\cS_{i,\hb}\subset \hQ$$
the subalgebra consisting of rational functions of the form 
 $A+B\prod\limits_{k=1}^l(a_k\al_i+\hb)$ 
 with $A,B\in\hQ$ both regular at $\hb\!=\!-a_k\al_i$
 for every $k\!\in\![l]$ and the denominator of $A$ an element of $\Q[\al_i,\hb]$.
We define
$$\iQ_{\hb_1,\hb_2}\subset \Q_{\al}(\hb_1,\hb_2)$$
to be the subalgebra generated by $\iQ_{\hb_1}$  and $\iQ_{\hb_2}$.
If in addition $j\!\in\![n]$, let 
\begin{equation*}\begin{split}
\K^{(i,j)} &\equiv 
\Span_{\Q}\big\{\al_i^{n-2}\cI\cdot\ti{\Q}_j[\al]^{S_{n-1}},
 \{1,\al_i,\ldots,\al_i^{n-3},\al_i^{n-1}\}\ti\Q_j[\al]^{S_{n-1}}\big\}
   \subset\Q_{\al},\\
\K^{(i,j)}_{\hb}&\equiv
\Span_{\Q}\big\{\al_i^{n-2}\cI\cdot\cS_{j,\hb},
\{1,\al_i,\ldots,\al_i^{n-3},\al_i^{n-1}\}\cS_{j,\hb}\big\}
\subset\Q_{\al}(\hb).
\end{split}\end{equation*}
All statements in the next lemma follow immediately from the definitions.


\begin{lmm}\label{algprop_lmm}
If $i\!\in\![n]$,
\begin{alignat*}{1}
F\in\hQ\big[\big[q\big]\big] \qquad&\Lra\qquad\fR_{\hb=0}F, 
\fR_{\hb=\i}F\in\ti\Q_i[\al]^{S_{n-1}}\big[\big[q\big]\big]; 
\\
F\in\iQ_{\hb_1,\hb_2}\big[\big[q\big]\big] \qquad&\Lra\qquad
\fR_{\hb_1=0}\fR_{\hb_2=0}\left\{\frac{F}{\hb_1+\hb_2}\right\}
\in\iQ\big[\big[q\big]\big];\,
\\
F\in\cS_{i,\hb}\big[\big[q\big]\big]\qquad&\Lra\qquad\fR_{\hb=-\a\al_i}
\left\{\frac{F}{\prod\limits_{r=1}^l(a_r\al_i+\hb)}\right\}
\in\iQ\big[\big[q\big]\big]. 
\end{alignat*}
If in addition $F,G\!\in\!q\hQ\big[\big[q\big]\big]$, then
\begin{equation*}\label{P5}
F-G\in \cI\cdot\hQ\big[\big[q\big]\big] ~~\Lra~~
\ne^F\!-\!\ne^G,\, \log(1\!+\!F)-\log(1\!+\!G)
\in \cI\cdot q\hQ\big[\big[q\big]\big].\,\,
\end{equation*}
\end{lmm}

\subsection{The genus zero generating functions}
\label{g0form_subs}

We will now express the genus 0 generating functions 
$\cZ_i^*$, $\cZ_{ij}^*$, and $\wt\cZ_{ii}^*$ defined in Section~\ref{equivcomp_subs}
in terms of the hypergeometric series~$\F$ of~\e_ref{F0def_e}
and the operator~$\bM$ of~\e_ref{Mdfn}. 

\begin{lmm}\label{Z_lmm}
The genus 0 generating functions $\cZ_i^*$, $\cZ_{ij}^*$, and $\wt\cZ_{ii}^*$ satisfy
\begin{alignat}{1}\label{Z1pt_lmm_e}
\left[(\al_i+\hb)^n-\al_i^n\right]\left[1\!+\!\cZ_i^*(\hb,Q)-\ne^{-J(q)\frac{\al_i}{\hb}}\frac{\F(\al_i/\hb,q)}{I_0(q)}\right]
~&\in \cI\cdot q\cS_{i,\hb}\big[\big[q\big]\big]\,,\\
\label{Z1pt2_lmm_e}
\left[(\al_j+\hb)^n-\al_j^n\right]\left[\al_i^{n-2}\al_j+\hb\cZ_{ji}^*(\hb,Q) 
-\al_i^{n-2}\al_j\ne^{-J(q)\frac{\al_j}{\hb}}\frac{\bM\F(\al_j/\hb,q)}{I_1(q)}\right]
~&\in\K^{(i,j)}_{\hb}\big[\big[q\big]\big]\,,\!\footnotemark\\
\label{Z2pt_lmm_e}
n\al_i^{n-1}+2\big(\hb_1\!+\!\hb_2\big)\hb_1\hb_2\wt\cZ_{ii}^*(\hb_1,\hb_2,Q)
\hspace{1.6in}& \notag\\
-\al_i^{n-1}\ne^{-J(q)\al_i\left(\frac{1}{h_1}+\frac{1}{\hb_2}\right)}
\bF(\al_i/\hb_1,\al_i/\hb_2,q)
\in\cI\cdot&\iQ_{\hb_1,\hb_2}\big[\big[q\big]\big],
\end{alignat}\footnotetext{Note that \e_ref{Z1pt2_lmm_e}
implies that 
\begin{equation*}\begin{split}
&\al_i^{n-2}\al_j+\hb\cZ_{ji}^*(\hb,Q) 
-\al_i^{n-2}\al_j\ne^{-J(q)\frac{\al_j}{\hb}}\frac{\bM\F(\al_j/\hb,q)}{I_1(q)}\\
&\hspace{1.8in}
\in  \Span_{\Q}\left\{\al_i^{n-2}\cI\cdot \ti\Q_j[\al]^{S_{n-1}}_{\hb},\{1,\al_i,\ldots,\al_i^{n-3},\al_i^{n-1}\}\ti\Q_j[\al]^{S_{n-1}}_{\hb}\right\}\big[\big[q\big]\big]\,;
\end{split}\end{equation*}
this is the only information about $\cZ^*_{ji}(Q)$ 
used in computing the non-equivariant part of $\ti\cA_{ij}(Q)$
in Section~\ref{mainthmpf_subs}.}
where
\begin{equation*}\begin{split}
\bF(w_1,w_2,q) &=\sum_{p=0}^{n-1-l}\frac{\bM^p\F(w_1,q)}{I_p(q)}
\frac{\bM^{n-1-l-p}\F(w_2,q)}{I_{n-1-l-p}(q)}\\
&\qquad +\sum_{p=1}^{l}\frac{\bM^{n-1-l+p}\F(w_1,q)}{I_{n-1-l+p}(q)}
\frac{\bM^{n-p}\F(w_2,q)}{I_{n-p}(q)}
\end{split}\end{equation*}
and $Q$ and $q$ are related by the mirror map \e_ref{mirmap_e}.
\end{lmm}

\begin{proof} By \cite[Theorem~11.8]{Gi},
\BE{Gi_e} 1\!+\!\cZ_i^*(\hb,Q)
=\ne^{\frac{-J(q)\al_i+C(q)\si_1}{\hb}}\frac{\Y(\hb,\al_i,q)}{I_0(q)}\EE
for some $C\!\in\!q\Q\big[\big[q\big]\big]$ and
$$\Y(\hb,x,q)\equiv \sum_{d=0}^{\i}q^d
\frac{\prod\limits_{k=1}^l\prod\limits_{r=1}^{a_kd}(a_kx\!+\!r\hb)}
{\prod\limits_{r=1}^d\left[\prod\limits_{k=1}^n(x\!-\!\al_k\!+\!r\hb)
-\prod\limits_{k=1}^n(x\!-\!\al_k)\right]}\,.$$
There is no term $\prod\limits_{k=1}^n(x\!-\!\al_k)$ in
the generating function used in place of $\Y$ in~\cite{Gi},
but putting it does not effect the validity of \cite[Theorem~11.8]{Gi}
as it vanishes under all evaluations $x\!\lra\!\al_i$.
On the other hand, with this extra term in place $\Y$ becomes a function of
$\hb$, $x$, $\si_1,\ldots,\si_{n-1}$, and {\it not}~$\si_n$.
Since all denominators in $\Y(\hb,\al_i,q)$ are products of 
$\al_i\!-\!\al_k\!+\!r\hb$ with $r\!\in\!\Z$, 
$$\Y(\hb,\al_i,q)-\F(\al_i/\hb,q)\in\cI\cdot\hQ\big[\big[q\big]\big].$$
If $a_k\!\neq\!1,2$ or $n$ is odd, the denominators in the above expression 
do not vanish at $\hb=-a_k\al_i$, and so the difference lies in 
$\cI\cdot q\cS_{i,\hb}\big[\big[q\big]\big]$.
Otherwise, the denominators have a simple zero at $\hb=-a_k\al_i$ 
(in $q$-degree at least 2 if $a_k\!=\!1$).
If $a_k\!=\!2$ and $n$ is even, the factor $[(\al_i+\hb)^n-\al_i^n]$ has a zero
at $\hb=-a_k\al_i$ as well, and so 
$$\left[(\al_i+\hb)^n-\al_i^n\right]\left[\Y(\hb,\al_i,q)-\F(\al_i/\hb,q)\right]
~\in \cI\cdot q\cS_{i,\hb}\big[\big[q\big]\big].$$
The case $a_k\!=\!1$ is excluded by the assumption on~$\a$ 
in Section~\ref{mirsym_sec}.
Thus, \e_ref{Z1pt_lmm_e} follows from \e_ref{Gi_e}.

By \e_ref{phidfn_e}, \cite[Theorem~4]{bcov0_ci}, and 
the same reasoning as in the previous paragraph,
\BE{Po_e1}\begin{split}
&\sum_{\begin{subarray}{c}p+r+s=n-1\\ p,r,s\ge0\end{subarray}}
\!\!\!\!\!\!\!(-1)^s\si_s\al_i^p\al_j^r
~+\hb\cZ_{ji}^*(\hb,Q)
 =\ne^{\frac{-J(q)\al_j+C(q)\si_1}{\hb}}
\!\!\!\!\!\!\sum_{\begin{subarray}{c}p+r+s=n-1\\ p,r,s\ge0\end{subarray}}
\!\!\!\!\!\!\!(-1)^s\si_s\al_i^p\Y_r(\hb,\al_j,q) 
\end{split}\EE
where $\Y_r\in\Q_{\al}(x)\big[\big[q\big]\big]$ is a power series such that 
\BE{Yr_e}\begin{split}\Y_r(\hb,\al_j,q)\in\cS_{j,\hb}\big[\big[q\big]\big]
\qquad\hbox{and}\qquad\qquad\qquad\qquad\qquad\\
\qquad\qquad\left[(\al_j+\hb)^n-\al_j^n\right]\left[Y_r(\hb,\al_j,q)-\al_j^r\frac{\bM^r\F(\al_j/\hb,q)}{I_r(q)}\right]
\in \cI\cdot q\cS_{j,\hb}\big[\big[q\big]\big].
~\footnotemark\end{split}\EE
\footnotetext{In the notation of~\cite{bcov0_ci}, 
$\Y_r(\hb,x,q)=\cY_r(\hb,x,q)/x^l$.}
The claim~\e_ref{Z1pt2_lmm_e} thus follows from \e_ref{Po_e1}.

Finally, by \cite[Theorem~4]{bcov0_ci},
\begin{gather}\label{Po_e}
\begin{split}
&\sum_{\begin{subarray}{c}p+r+s=n-1\\ p,r,s\ge0\end{subarray}}
\!\!\!\!\!\!\!\!(-1)^s\si_s\al_i^{p+r} ~+
2\big(\hb_1\!+\!\hb_2\big)\hb_1\hb_2\wt\cZ_{ii}^*(\hb_1,\hb_2,Q)\\
&\hspace{.5in}
=\ne^{(-J(q)\al_i+C(q)\si_1)\left(\frac{1}{\hb_1}+\frac{1}{\hb_2}\right)}
\al_i^l \!\!\!\!\!\!
\sum_{\begin{subarray}{c}p+r+s=n-1\\ p,r,s\ge0\end{subarray}}
\!\!\!(-1)^s\si_s\Y_p(\hb_1,\al_i,q)\Y_{r-l}(\hb_2,\al_i,q),
\end{split}\\
\label{Po_e2}
\sum_{\begin{subarray}{c}s_1+s_2=k\\ 0\leq s_1\leq n\end{subarray}}\!\!\!
(-1)^{s_1}\si_{s_1}\Y_{s_2}(\hb,\al_i,q)=0  \qquad\hbox{if}\quad n\!-\!l\leq k\leq n\!-\!1,
\end{gather}
where $\Y_r\in\Q_{\al}(x)\big[\big[q\big]\big]$ is a power series such that 
\BE{Yr_e2} \Y_r(\hb,\al_i,q)\in\ti\Q_i[\al]_{\hb}^{S_{n-1}}\big[\big[q\big]\big].\EE
By \e_ref{Po_e2} and \e_ref{Yr_e2}, 
$$\Y_{-p}(\hb,\al_i,q)-\al_i^{-n}\Y_{n-p}(\hb,\al_i,q)
\in \cI\cdot\ti\Q_i[\al]_{\hb}^{S_{n-1}}\big[\big[q\big]\big]
\qquad\forall~p=1,2,\ldots,l.$$
Thus, \e_ref{Z2pt_lmm_e} follows from \e_ref{Po_e}, \e_ref{Yr_e2},
and \e_ref{Yr_e}.
\end{proof}

\subsection{Proof of Theorem~\ref{main_thm}}
\label{mainthmpf_subs}

We will use Lemmas~\ref{algebr_lmm1}-\ref{Z_lmm}
to extract the coefficients of $\al_i^{n-2}$  from the expressions of 
Proposition~\ref{equivred_prp}  modulo~$\K_i\big[\big[q\big]\big]$.
In the notation of Theorem~\ref{main_thm} and Proposition~\ref{equivred_prp}, 
\begin{alignat}{1}\label{Acoeff_e}
Q\frac{\nd\ti{A}(q)}{\nd Q} &= ~\hbox{the coefficient of}~ 
\al_i^{n-2} ~\hbox{in}~ \left(\cA_i(Q)+\sum_{j=1}^n\ti{\cA_{ij}}(Q)\right)
~\hbox{modulo}~ \K_i\big[\big[q\big]\big],\\
\label{Bcoeff_e}
Q\frac{\nd\ti{B}(q)}{\nd Q} &= ~\hbox{the coefficient of}~
\al_i^{n-2} ~\hbox{in}~ \left(\cB_i(Q)+\sum_{j=1}^n\ti{\cB_{ij}}(Q)\right)
~\hbox{modulo}~ \K_i\big[\big[q\big]\big].
\end{alignat}
Let $D\equiv q\frac{\nd}{\nd q}$ as in Section~\ref{hg_sec}.
We begin by computing residues of the transforms of~$\F$ that appear 
in the description of the generating functions in Lemma~\ref{Z_lmm}.

\begin{lmm}\label{res_lmm}
With $\mu,\Phi_0,L,\bA\in\Q\big[\big[q\big]\big]$ given by \e_ref{asym}, \e_ref{Ldfn_e0},
and \e_ref{main_thm_e1},
\begin{alignat}{1}\label{Zijres_e}
\fR_{\hb=0}\left\{\hb^{-1}\ne^{-\mu(q)\frac{\al_j}{\hb}}\bM\F(\al_j/\hb,q)\right\}
&=\frac{L(q)\Phi_0(q)}{I_0(q)}\,,\\
\label{Ztires_e}
\fR_{h_1=0}\fR_{h_2=0}\left\{
\frac{\ne^{-\mu(q)\al_i(\hb_1^{-1}+\hb_2^{-1})}}{\hb_1\hb_2(\hb_1\!+\!\hb_2)}
\bF(\al_i/\hb_1,\al_i/\hb_2,q)\right\}&= \al_i^{-1}L(q)^{-1}D\bA(q).
\end{alignat}
\end{lmm}

\begin{proof} By Lemmas~\ref{HGper_lmm} and~\ref{reg_lmm}, \e_ref{asym}, \e_ref{Mdfn}, and the first statement 
in \e_ref{muPhi0_e}, 
\BE{asym_p2} 
\bM^p\F(w,q)\sim\ne^{\mu(q)w}\sum_{s=0}^{\i}\Phi_{p,s}(q)\,w^{-s} \qquad(w\to\i), \EE
where
$$ \Phi_{0,s}=\Phi_s\,, \qquad
\Phi_{p+1,s}=\frac{L}{I_p}\,\Phi_{p,s}\,+\, 
\begin{cases} D\left(\frac{\Phi_{p,s-1}}{I_p}\right),&\hbox{if}~s\ge1,\\ 
 0,&\hbox{if}~s=0. \end{cases}$$
The $s\!=\!0,1$ cases of the recursion give
\BE{expind_e2}  
\Phi_{p,0}=H_p\Phi_0, \qquad 
\Phi_{p,1}=H_p\left(\Phi_1+pD\bigg(\frac{\Phi_0}{L}\bigg)+
\frac{\Phi_0}{L}\sum_{r=1}^p\frac{DH_r}{H_r}\right)\EE
where $H_p=L^p/(I_0\ldots I_{p-1})$
as in Section~\ref{muPhi0_subs}.
The $p\!=\!1$ case of the first identity above gives~\e_ref{Zijres_e}.

On the other hand, for all $p,r\!\ge\!0$
\begin{equation*}\begin{split}
&\fR_{h_1=0}\fR_{h_2=0}\left\{
\frac{\ne^{-\mu(q)\al_i(\hb_1^{-1}+\hb_2^{-1})}}{\hb_1\hb_2(\hb_1\!+\!\hb_2)}
\bM^p\F(\al_i/\hb_1,q)\bM^r\F(\al_i/\hb_2,q)\right\}\\
&\quad =\fR_{h_1=0}\left\{\frac{\ne^{-\mu(q)\al_i/\hb_1}}{\hb_1^2}
\bM^p\F(\al_i/\hb_1,q)\cdot\Phi_{r,0}(q)\right\}
=\al_i^{-1}\Phi_{p,1}(q)\Phi_{r,0}(q).
\end{split}\end{equation*}
By \e_ref{HG1e0}-\e_ref{HG1e1},
\begin{alignat*}{2}
H_pH_{n-1-l-p}=\frac{I_pI_{n-1-l-p}}{L^{l+1}} &\qquad &\hbox{if}~~
 0\le p\le n\!-\!1\!-\!l, \\
H_{n-1-l+p}H_{n-p}=\frac{I_{n-1-l+p}I_{n-p}}{L^{l+1}} &\qquad& \hbox{if}~~
 1\le p\le l.
\end{alignat*}
Thus, by~\e_ref{expind_e2} and the second statement in \e_ref{muPhi0_e},
\begin{equation*}\begin{split}
&\fR_{h_1=0}\fR_{h_2=0}\left\{
\frac{\ne^{-\mu(q)\al_i(\hb_1^{-1}+\hb_2^{-1})}}{\hb_1\hb_2(\hb_1\!+\!\hb_2)}
\bF(\al_i/\hb_1,\al_i/\hb_2,q)\right\}\\
&\hspace{1in}
=\al_i^{-1}L^{-\frac{l+1}{2}}\Bigg(n\Phi_1+\binom{n}{2}D\big(L^{\frac{l-1}{2}}\big)
+L^{\frac{l-1}{2}}\sum_{r=1}^{n-1}(n\!-\!r)\frac{DH_r}{H_r}\Bigg)
\end{split}\end{equation*}
By \e_ref{HG1e1} and \e_ref{HG1e2},
\begin{equation*}\begin{split}
&\sum_{r=1}^{n-1}(n\!-\!r)\frac{DH_r}{H_r}
=\sum_{r=1}^{n-l}(n\!-\!l\!-\!r)\frac{DH_r}{H_r}
+l\sum_{r=1}^{n-l}\frac{DH_r}{H_r}
+\sum_{r=1}^{l-1}(l\!-\!r)\frac{DH_{n-l+r}}{H_{n-l+r}}\\
&\qquad
=\binom{n\!-\!l\!+1}{3}\frac{DL}{L}
-\sum_{r=1}^{n-l}\sum_{p=0}^{r-1}(n\!-\!l\!-\!r)\frac{DI_p}{I_p}
-\frac{(l\!-\!1)l(n\!-\!l)}{2}\frac{DL}{L}
-\frac{(l\!-\!1)l(2l\!-\!1)}{6}\frac{DL}{L}\\
&\qquad
=\binom{n\!-\!l\!+1}{3}\frac{DL}{L}-\frac{(l\!-\!1)l(3n\!-\!l\!-\!1)}{6}\frac{DL}{L}
-\sum_{p=0}^{n-l}\binom{n\!-\!l\!-\!p}{2}\frac{DI_p}{I_p}\,.
\end{split}\end{equation*}
The second identity in the lemma follows from the last two equations
along with \e_ref{muPhi0_e}, \e_ref{Phi1_e}, $DL\!=\!L(L^n\!-\!1)/n$,
and \e_ref{Ldfn_e0}.
\end{proof}

We will now compute \e_ref{Acoeff_e}.
By Lemma~\ref{algprop_lmm}, \e_ref{Z1pt_lmm_e}, \e_ref{etadfn_e}, \e_ref{Phi0dfn_e},  
and~\e_ref{asym}, 
\begin{alignat}{1} \label{etadiff_e}
\eta_i(Q)-\big(\mu(q)-J(q)\big)\al_i ~&\in  
\cI\cdot q\iQ\big[\big[q\big]\big]\,,\\
\label{Phi0diff_e}
\Phi_0(\al_i,Q)-\frac{\Phi_0(q)}{I_0(q)}
~&\in  \cI\cdot q\iQ\big[\big[q\big]\big].
\end{alignat}
Thus, by Proposition~\ref{equivred_prp}, Lemma~\ref{algprop_lmm},
\e_ref{Z2pt_lmm_e}, and~\e_ref{Ztires_e},
\BE{Adiff_e}
\cA_i(Q)-\al_i^{n-2}\frac{I_0(q)D\bA(q)}{2L(q)\Phi_0(q)} 
\in\cI\cdot\iQ\big[\big[q\big]\big]. \EE
By Proposition~\ref{equivred_prp}, Lemma~\ref{algprop_lmm}, 
\e_ref{Z1pt2_lmm_e}, \e_ref{Zijres_e}, \e_ref{etadiff_e} and \e_ref{Adiff_e}, 
\BE{tiAdiff_e1}
\sum_{j=1}^n\ti\cA_{ij}(Q)
-\sum_{j=1}^n\left(\frac{\al_j^{n-2}I_0(q)D\bA(q)/(2L(q)\Phi_0(q))}
    {\prod\limits_{k\neq j}(\al_j\!-\!\al_k)}  
    \left[\frac{L(q)\Phi_0(q)}{I_0(q)I_1(q)}-1\right]\al_i^{n-2}\al_j\right)
\in\K_i\big[\big[q\big]\big].\EE
By the Residue Theorem on~$S^2$,
\BE{resth_e}
\sum_{j=1}^n\frac{\al_j^{n-1}}{\prod\limits_{k\neq j}(\al_j\!-\!\al_k)} 
=\sum_{j=1}^n\fR_{z=\al_j}\left\{\frac{z^{n-1}}{\prod\limits_{k=1}^n(z\!-\!\al_k)}\right\}
=-\fR_{z=\i}\left\{\frac{z^{n-1}}{\prod\limits_{k=1}^n(z\!-\!\al_k)}\right\}
=1.\EE
Thus, by~\e_ref{Adiff_e} and \e_ref{tiAdiff_e1},
$$\cA_i(Q)+\sum_{j=1}^n\ti\cA_{ij}(Q)-\al_i^{n-2}\frac{D\bA(q)}{2I_1(q)}
\in\K_i\big[\big[q\big]\big].$$
Since $\frac{1}{I_1(q)}D=Q\frac{\nd}{\nd Q}$, this proves 
the claim stated in the sentence after Theorem~\ref{main_thm}.

We next compute \e_ref{Bcoeff_e}. Let
\BE{tibB_e} \ti\bB(q)\equiv -\frac{24}{\lr{\a}}\ti{B}(q)\EE
with $\ti{B}(q)$ as in Theorem~\ref{main_thm} and
\begin{equation*}\begin{split}
\al_j^{n-2-l}B(q)&\equiv
\fR_{\hb=0,\i,-\a\al_j}\left\{
\frac{(\al_j\!+\!\hb)^n-\al_j^n}{\hb^3\prod\limits_{r=1}^l(a_r\al_j\!+\!\hb)}
\frac{\ne^{-J(q)\frac{\al_j}{\hb}}\F(\al_j/\hb,q)/I_0(q)-1}
{\ne^{-J(q)\frac{\al_j}{\hb}}\F(\al_j/\hb,q)/I_0(q)}\right\}
\end{split}\end{equation*}

\begin{lmm}\label{res_lmm2}
With notation as above,
\begin{alignat}{1}\label{Bres_e1}
\fR_{\hb=0,\i,-\a\al_j}\left\{
\frac{(\al_j\!+\!\hb)^n-\al_j^n}{\hb^3\prod\limits_{r=1}^l(a_r\al_j\!+\!\hb)}
\frac{\ne^{-J(q)\frac{\al_j}{\hb}}\frac{\bM\F(\al_j/\hb,q)}{I_1(q)}-1}
{\ne^{-J(q)\frac{\al_j}{\hb}}\F(\al_j/\hb,q)/I_0(q)}\right\}
&=\al_j^{n-2-l}\left(Q\frac{\nd\ti\bB}{\nd Q}(q)+B(q)\right);\\
\label{Bres_e2}
\fR_{\hb=0,\i,-\a\al_j}\left\{
\frac{(\al_j\!+\!\hb)^n-\al_j^n}{\hb^2\prod\limits_{r=1}^l(a_r\al_j\!+\!\hb)}
\log\left(\ne^{-J(q)\frac{\al_j}{\hb}}\frac{\F(\al_j/\hb,q)}{I_0(q)}\right)\right\}
&=\al_j^{n-1-l}\ti\bB(q).
\end{alignat}
\end{lmm}

\begin{proof} Since $I_1(q)=1+q\frac{\nd}{\nd q}J(q)$
and $Q\frac{\nd}{\nd Q}=\frac{1}{I_1}q\frac{\nd}{\nd q}$, 
\begin{equation*}\begin{split}
\ne^{-J(q)w}\frac{\bM\F(w,q)}{I_1(q)}
&=\frac{1}{I_1(q)}\left\{1+q\frac{\nd J}{\nd q}(q)+\frac{q}{w}\frac{\nd}{\nd q}\right\}
\left(\ne^{-J(q)w}\frac{\F(w,q)}{I_0(q)}\right)\\
&=\left\{1+\frac{Q}{w}\frac{\nd}{\nd Q}\right\}\left(\ne^{-J(q)w}\frac{\F(w,q)}{I_0(q)}\right).
\end{split}\end{equation*}
Along with the definition of $B$ and \e_ref{Bres_e2}, this gives \e_ref{Bres_e1}.

By the Residue Theorem on $S^2$, the terms $\ne^{-J(q)\al_j/\hb}$
and $I_0(q)$ do not effect the left-hand side of~\e_ref{Bres_e2}.
Since the coefficients of the positive powers of 
$\left[(\al_j+\hb)^n-\al_j^n\right]\F(\al_j/\hb,q)$
vanish at $\hb\!=\!-a_k\al_j$ to the same order as 
$\prod\limits_{r=1}^l(a_r\al_j\!+\!\hb)$,\footnote{because $a_k\neq1$ by assumption}
it follows that 
\begin{equation*}\begin{split}
&\fR_{\hb=0,\i,-\a\al_j}\left\{
\frac{(\al_j\!+\!\hb)^n-\al_j^n}{\hb^2\prod\limits_{r=1}^l(a_r\al_j\!+\!\hb)}
\log\left(\ne^{-J(q)\frac{\al_j}{\hb}}\frac{\F(\al_j/\hb,q)}{I_0(q)}\right)\right\}\\
&\hspace{2in}
=\fR_{\hb=0,\i}\left\{
\frac{(\al_j\!+\!\hb)^n-\al_j^n}{\hb^2\prod\limits_{r=1}^l(a_r\al_j\!+\!\hb)}
\log \F(\al_j/\hb,q)\right\}.
\end{split}\end{equation*}
Expanding $(\al_j\!+\!\hb)^n-\al_j^n$ and using \e_ref{asym}, we obtain
\begin{equation*}\begin{split}
\fR_{\hb=0}\left\{
\frac{(\al_j\!+\!\hb)^n-\al_j^n}{\hb^2\prod\limits_{r=1}^l(a_r\al_j\!+\!\hb)}
\log \F(\al_j/\hb,q)\right\}
=\frac{\al_j^{n-1-l}}{\lr{\a}}\left(
\left[\binom{n}{2}-n\sum_{r=1}^l\frac{1}{a_r}\right]\mu(q)
+n\log\Phi_0(q)\right).
\end{split}\end{equation*}
On the other hand,
\begin{equation*}\begin{split}
\fR_{\hb=\i}\left\{
\frac{(\al_j\!+\!\hb)^n-\al_j^n}{\hb^2\prod\limits_{r=1}^l(a_r\al_j\!+\!\hb)}
\log \F(\al_j/\hb,q)\right\}
&=-\fR_{w=0}\left\{
\frac{(\al_jw\!+\!1)^n-\al_j^nw^n}{w^{n-l}\prod\limits_{r=1}^l(1\!+\!a_r\al_jw)}
\log \F(w\al_j,q)\right\}\\
&=-\al_j^{n-1-l}\sum_{p=0}^{n-1-l}
\left\llbracket
\frac{(1\!+\!w)^n}{\prod\limits_{k=1}^l(1\!+\!a_k w)}\right\rrbracket_{w;n-1-l-p}
\hspace{-.5in}\left\llbracket\log\F(w,q)\right\rrbracket_{w;p}\,.
\end{split}\end{equation*}
Since $\lrbr{\F(w,q)}_{w;0}\!=\!I_0(q)$ and 
$\lrbr{\F(w,q)}_{w;1}\!=\!J(q)$,
\e_ref{Bres_e2} follows by adding up the last two equations and using
\e_ref{chclass_e}, the second identity in \e_ref{muPhi0_e}, and \e_ref{Ldfn_e0}.
\end{proof}

We now complete the proof of Theorem~\ref{main_thm}.
By Proposition~\ref{equivred_prp}, Lemma~\ref{algprop_lmm},
\e_ref{Z1pt_lmm_e}, and the definition of $B(q)$ above,
$$\cB_i(Q)-\al_i^{n-2}\frac{\lr{\a}}{24}B(q)
\in\cI\cdot\iQ\big[\big[q\big]\big]. $$
By Proposition~\ref{equivred_prp}, Lemma~\ref{algprop_lmm}, 
\e_ref{Z1pt2_lmm_e}, and~\e_ref{Bres_e1}, 
$$\sum_{j=1}^n\ti\cB_{ij}(Q)
+\frac{\lr{\a}}{24}\sum_{j=1}^n
 \left(\frac{\al_j^{n-2}}{\prod\limits_{k\neq j}(\al_j\!-\!\al_k)}  
    \left[\left(Q\frac{\nd\ti\bB}{\nd Q}(q)+B(q)\right) \right]\al_i^{n-2}\al_j\right)
\in\K_i\big[\big[q\big]\big].$$
By the last two equations, \e_ref{resth_e}, and \e_ref{tibB_e}, 
$$\cB_i(Q)+\sum_{j=1}^n\ti\cB_{ij}(Q)
-\al_i^{n-2}Q\frac{\nd \ti{B}(q)}{\nd Q}\in  \K_i\big[\big[q\big]\big].$$
This concludes the proof of  Theorem~\ref{main_thm}.

\vspace{.2in}

\noindent
{\it Department of Mathematics, SUNY Stony Brook, NY 11794-3651\\
alexandra@math.sunysb.edu}

\end{document}